\documentclass[10pt,twoside,a4paper]{amsart} 
\usepackage{amsmath,amssymb}
\usepackage{color}
\usepackage{graphicx}
\usepackage{amsthm}
\usepackage[all]{xy}
\usepackage{verbatim}
\usepackage{enumerate}
\usepackage{caption}
\usepackage{pgf,tikz}
\usetikzlibrary{arrows}

\theoremstyle{definition}
\newtheorem{definition}{Definition}[section]
\newtheorem{remark}[definition]{Remark}

\theoremstyle{plain}
\newtheorem{theorem}[definition]{Theorem}
\newtheorem{proposition}[definition]{Proposition}

\newtheorem{corollary}[definition]{Corollary}
\newtheorem{lemma}[definition]{Lemma}

\def\dim{\rm dim}

\def\be{\begin{equation}}
\def\ee{\end{equation}}
\def\ben{\begin{displaymath}}
\def\een{\end{displaymath}}
\def\baa{\begin{eqnarray}}
\def\eaa{\end{eqnarray}}

\def\ba{\begin{array}}
\def\ea{\end{array}}
\def\la{\label}
\def\p{\partial}
\def\nd{w}

\makeatletter
\@addtoreset{equation}{section}
\makeatother

\def\CC{{\mathcal C}}
\def\Ch{\widehat {C}}

\def\nd{w}

\def\C{{\mathbb C}}
\def\R{{\mathbb R}}
\def\Z{{\mathbb Z}}

\def\Mcal{{\mathcal M}}

\def\Hc{{\mathcal H}}

\def\bk{{\bf k}}

\def\Pic{{\rm Pic}}
\def\rk{{\rm rk} \,}

\def\Acal{{\mathcal A}}

\def\Cb{\overline{{\mathcal C}}}
\def\M{{\mathcal M}}
\def\Mb{\overline{{\mathcal M}}}

\def\2x2{{\left(\!\!\begin{array}{cc}a&b\\c&d\\\end{array}\!\!\right)}}

\def\f{\frac}
\def\e{\epsilon}

\def\a{\alpha}
\def\b{\beta}

\def\deg{{\rm {\rm deg}}}
\def\O{\Omega}

\def\Pcal{{\mathcal P}}

\def\gh{\hat{g}}

\def\xh{\hat{x}}

\def\ka{\kappa}

\def\L{\Lambda}
\def\l{\lambda}

\def\dim{{\rm dim}}

\def\gh{\widehat{g}}

\def\Mgn{\mathfrak{M}_{g}^{(n)}}
\def\Mgno{\overline{\mathfrak{M}}_{g}^{(n)}}

\def\H1c{H^{(1)}}
\def\H1h{H_{(1)}}
\def\H0c{H^{(0)}}
\def\H0h{H_{(0)}}

\def\pt{\lambda_{PT}}

\def\dim{{\rm dim}}

\begin{document}

\title[Tau functions, Prym-Tyurin classes and degenerate differentials]{Tau functions, Prym-Tyurin classes and loci of degenerate differentials}
\author{Dmitri Korotkin}
\address{Department of Mathematics and Statistics, Concordia University, 1455 de Maisonneuve W., Montreal, Qu\'ebec, 
Canada H3G 1M8}
\email{korotkin@mathstat.concordia.ca}
\author{Adrien Sauvaget}
\address{IMJ-PRG, Universit\'e Pierre et Marie Curie\\ 4 place Jussieu\\ 75005 Paris, France}
\email{adrien.sauvaget@imj-prg.fr}
\author{Peter Zograf}
\address{Steklov Mathematical Institute, Fontanka 27, Saint Petersburg 191023
Russia, and Chebyshev Laboratory, Saint Petersburg State University, 14-th Line
V.O. 29, Saint Petersburg 199178 Russia}
\email{zograf@pdmi.ras.ru}

\date{\today}

\begin{abstract}
We study the rational Picard group of the projectivized moduli space $P\Mgno$ of holomorphic $n$-differentials on complex genus $g$ stable curves. We define
$n-1$ natural classes in this Picard group that we call {\em Prym-Tyurin} classes. We express these classes as linear combinations of boundary divisors and the divisor of $n$-differentials with a double zero. We give two different proofs of this result, using two alternative approaches: an analytic approach that involves the Bergman tau function and its vanishing divisor and an algebro-geometric approach that involves cohomological computations on the universal curve. 
\end{abstract}

\maketitle

\tableofcontents

\section{Introduction}

\subsection{Moduli space of $n$-differentials} 

Let $g$ and $n$ be positive integers with $g\geq 2$. Let $\M_g$ (respectively $\Mb_g$) be the moduli space of smooth  (respectively stable nodal) complex curves. Denote by $D_0 \subset \Mb_g$ the closure of the locus of stable curves with one nonseparating node. Further, denote by $D_i \subset \Mb_g$, $1 \leq i \leq [g/2 ]$, the closure of the locus of curves with a separating node and two irreducible components of genera $i$ and $g-i$. Finally, denote by $\pi: \CC_g \to \M_g$ or $\pi:\Cb_g\to \Mb_g$ the universal curve and by $\omega = \omega_{\Cb_g/\Mb_g}$ the relative dualizing sheaf. 

Let 
$$
\O^{(n)}_g = R^0 \pi_* \omega^{\otimes n}
$$
be the direct image of the $n$th tensor power of $\omega$. Using the Riemann-Roch formula and Serre's duality, one can easily check that $h^1(C,\omega^{\otimes n}_C)=0$ if $n\geq 2$ and $h^1(C,\omega_C)=1$ for any stable curve $C$. Thus $\O^{(n)}_g$ is a vector bundle for any $n \geq 1$. For $n=1$ we write $\O_g$ instead of $\O^{(1)}_g$ and call it the {\em Hodge bundle}. The Riemann-Roch formula implies:
$$
\rk \, \O^{(n)}_g=\left\{ \begin{array}{cr}g& \text{if $n=1$;}\\(2n-1)(g-1)& {\rm otherwise.}\end{array}\right.
$$

We define the following elements in the rational Picard group of $\Mb_g$:
\begin{itemize}
\item the Chern class $\lambda_n \in \Pic(\Mb_g)$ of the determinant line bundle of $\O_g^{(n)}$;
\item the Poincar\'e dual classes $\delta_0,\ldots,\delta_{[g/2]} \in \Pic(\Mb_g)$ of the boundary divisors $D_0,D_1,\ldots,D_{[g/2]} \subset \Mb_{g}$.
\end{itemize}
For $k=1$ we write $\lambda$ instead of $\lambda_1$ and call this class the {\em Hodge class}.

\begin{definition}
The total space of the vector bundle $\O^{(n)}_g$ is denoted by $\Mgno$ and is called the {\em space of $n$-differentials}.
\end{definition}

The points of $\Mgno$ correspond to equivalence classes of pairs $(C,\nd)$, where $C$  is a stable genus $g$ algebraic curve, and $\nd$ is an $n$-differential on $C$. We recall that an $n$-differential $\nd$ on $C$ is a meromorphic $n$-differential on each irreducible component of the normalization of $C$ such that
\begin{itemize}
\item $\nd$ can only have poles at the preimages of the nodes;
\item these poles are of order at most~$n$;
\item at every node the $n$-residues of $\nd$ at the poles satisfy
\begin{equation*}
{\rm{res}}_{p_1}(\nd) = (-1)^n{\rm{res}}_{p_2}(\nd)
\end{equation*}
where $p_1$ and $p_2$ are the two preimages of the node.
\end{itemize}

We denote by $\nu:\Mgno\to \Mb_g$ the forgetful map and we will use the same notation for its restriction $\nu:\Mgn\to \M_g$ to the locus of smooth curves. We also denote by $\tilde{\nu}:P\Mgno\to \Mb_g$ the projectivized space of $n$-differentials.

In this paper we study the Picard group of $P\Mgno$. We will work over rational numbers, so that $\Pic$ will always denote the {\em rational} Picard group.

By abuse of notation we will denote by $\lambda$, $\lambda_n$ and $\delta_i$ both the elements of $\Pic(\Mb_g)$ and their pull-backs in $\Pic(P\Mgno)$. In addition, we  introduce the first Chern class $\psi \in \Pic(P\Mgno)$ of the tautological line bundle $L\to P\Mgno$.

The following lemma is standard (cf., e.g, \cite{MRL}, Lemma 1).
\begin{lemma}\label{pic}
The classes $\lambda, \psi,\delta_0,\ldots,\delta_{[ g/2]}$ form a basis of $\Pic(P\Mgno)$.
\end{lemma}

The goal of this paper is to define the Prym-Tyurin classes in  $\Pic(P\Mgno)$ and express them in the above basis.

\subsection{Stratification of $P\Mgn$}
\label{Ssec:strata}

The space of $n$-differentials is naturally stratified according to the multiplicities of the differential's zeros.

Let $\bk=(k_1,\ldots,k_m)$ be a partition of $n(2g-2)$. We denote by $\Mgn[\bk]\subset \Mgn$ the locus of pairs $(C,\nd)$ such that the $n$-differential $\nd$ has $m$ pairwise distinct zeros of orders exactly $k_i$. This locus is $\C^*$-invariant, thus we can also define its projectivization $P\Mgn[\bk]\subset P\Mgn$.  The space $P\Mgn$ is the disjoint union of the strata $P\Mgn[\bk]$ for all partitions $\bk$ of $n(2g-2)$. The following properties of the strata were proved in~\cite{Lan} and~\cite{Sch}.
\begin{itemize}
\item Each stratum $P\Mgn[\bk]$ is smooth. 
\item If at least one $k_i$ is not divisible by $n$ then either $P\Mgn[\bk]$ is empty or it has pure dimension $2g-3+m$.
\item If all $k_i$'s are divisible by~$n$ then $P\Mgn[\bk]$ has at least one irreducible component of dimension $2g-2+m$. A differential $(C, \nd)$ lies in a component like that if and only if $\nd$ is the $n$th power of a holomorphic differential. The stratum $P\Mgn[\bk]$ may also have irreducible components of dimension $2g-3+m$ composed of $n$-differentials that are not $n$th powers.
\end{itemize}
We denote by $P\Mgno[\bk]$ the closure of $P\Mgn[\bk]$ in $P\Mgno$. In particular we have $P\Mgno[\mathbf{1}] = P \Mgno$, where $\mathbf{1}$ stands for the partition $(1,1,\ldots,1)$.
\begin{definition} \label{Def:degenerate}
Let $g,n\geq 2$. The {\em divisor of degenerate $n$-differentials} is defined as 
$$
D_{\rm deg}= 
\left\{
\begin{array}{lll}
P\Mgno[2,1,\ldots,1], & \mbox{ if } & (g,n) \ne (2,2),\\
P\Mgno[2,1,1] + 2 \cdot P \Mgno[2,2], & \mbox{ if } & g=n=2.
\end{array} 
\right.
$$
We denote by $\delta_{\rm deg}$ the cohomology class that is Poincar\'e dual of $D_{\rm deg}$.
\end{definition} 

\begin{remark} Heuristically, $D_{\rm deg}$ is the divisor of $n$-differentials with a double zero, and for $(g,n)\neq (2,2)$ it is just the closure of 
$P\Mgn[2,1,\ldots,1]$ in $P\Mgno$. In the case $g=n=2$, however, $D_{\rm deg}$ has a special component consisting of squares of holomorphic differentials. This is because in genus 2 each quadratic differential is invariant with respect to the hyperelliptic involution. The four simple zeroes of $w$ are pairwise equivalent under the hyperelliptic involution, and when two non-equivalent zeroes coalesce, the other two ones also coalesce, giving a differential with two double zeroes. Since every such differential has two square roots that differ by a sign, the divisor $P\Mgno[2,2]$ comes with a factor of 2. (Note that when two equivalent zeroes coalesce, the differential in the limit has one double zero at a Weierstrass point and two simple zeroes.) 
\end{remark}

\subsection{First definition of Prym-Tyurin classes}\label{introfirst}
\label{Ssec:IntroPrym}

Let $(C,\nd)$ be a point in the projectivized moduli space $P\Mgn[\mathbf{1}]$. 
One can define a canonical cyclic ramified covering $f:\Ch\to C$ of degree $n$, where
$$\Ch=\{(x,v)|x\in C,\;v\in T_x^*C,\; v^n=\nd\}.$$  
This covering is completely ramified over the zeros of $\nd$. The curve $\Ch$ is smooth of genus $\gh=n^2(g-1)+1$. 
It comes with a {\em canonical holomorphic differential} $v$ given by $v(x,v)=v$. This differential $v$ on $\Ch$ satisfies $v^n=f^*\nd$. 
 
The action of $\Z/n\Z$ on the covering is given by $\rho^{k}:(x,v)\mapsto (x,\rho^k v)$, where $\rho=e^{\frac{2\pi\sqrt{-1}}{n}}$. We denote by $\sigma:\Ch\to\Ch$ the automorphism of $\Ch$ corresponding to $k=1$. Now consider the natural map 
\begin{align*}
\hat{\nu}  :  P\Mgn[\mathbf{1}] &\to  \M_{\gh},\\
(C,\nd) & \mapsto  \Ch
\end{align*}
($\Ch$ remains the same when we multiply $\nd$ by a non-zero constant). We consider the pull-back of the Hodge bundle $\Omega_{\gh}$ by the map $\hat{\nu}$. The automorphism $\sigma$ induces an endomorphism $\sigma^*$ of the vector bundle $\hat{\nu}^*\Omega_{\gh}$ given by:  $\left((C,\nd), u\right)\mapsto \left((C,\nd),\sigma^*u\right)$, where $u$ is an element of $H^0(\Ch,\omega_{\Ch})$. The endomorphism $\sigma^*$ satisfies $(\sigma^*)^n={\rm Id}$. Hence we have a decomposition 
\be\label{decompo}
\hat{\nu}^*\Omega_{\gh}=\bigoplus_{k=0}^{n-1}\Lambda^{(k)},
\ee 
where $\Lambda^{(k)}$ is the eigenbundle of $\hat{\nu}^*\Omega_{\gh}$ corresponding to the eigenvalue $\rho^k=e^{\frac{2\pi\sqrt{-1} k}{n}}$. 
\begin{remark}
The space $\Lambda^{(k)}$ is a vector bundle because the dimension of the fiber of $\Lambda^{(k)}$ is upper-continuous for all $k$ and $\hat{\nu}^*\Omega_{\gh}$ is a vector bundle thus the rank of each $\Lambda^{(k)}$ is constant.
\end{remark}

\begin{definition}
The vector bundles $\Lambda^{(k)}$ are called the {\em Prym-Tyurin vector bundles}. The {\em Prym-Tyurin class} $\pt^{(k)}$ is the first Chern class $c_1(\Lambda^{(k)}) \in \Pic(P\Mgn[\mathbf{1}])$. 
\end{definition}
For $n=2$ the study of vector bundles of this type was initiated by Prym~\cite{Prym} and for $n>2$ by A.~N.~Tyurin~\cite{Tyurin}. 

\begin{remark}
By abuse of notation we denote in the same way the {\em determinant line bundle} $\pt^{(k)}=\det \Lambda^{(k)}$ and its class in the Picard group.
\end{remark}

We will see in Section~\ref{PTND} that the map $\hat{\nu}: \Mgn \to \M_{\hat{g}}$, used to define the Prym-Tyurin vector bundles, admits no natural extension to $P\Mgno$. Nonetheless, in the next section we extend the definition of the Prym-Tyurin class to $P\Mgno$ by a construction involving the space of admissible covers and an intermediate bigger stack.

\subsection{Admissible coverings and Prym-Tyurin bundles}


Let $N=2n(g-1)$ be the degree of $\omega^{\otimes n}$.
We denote by ${\rm Hur}_g^n$ the moduli space whose geometric points are isomorphism classes of pairs $(f:\Ch\to C, \sigma)$ where:
\begin{itemize}
\item $C$ and $\Ch$  are smooth curves;
\item $f$ is a cyclic ramified covering of degree $n$ which is totally ramified over $N$ distinct points of~$C$;
\item $\sigma$ is an automorphism of $\Ch$ that commutes with $f$. 
\end{itemize}

We denote by $\overline{\rm Hur}^n_g$ the compactification of this space by admissible coverings (see~\cite{HM}). The space of admissible coverings has two forgetful maps (source and target of the covering):
$$
\xymatrix{
&\overline{\rm Hur}^n_g \ar[ld]_{{\rm target}} \ar[rd]^{{\rm source}}&\\
\Mb_{g,N}/S_N && \Mb_{\gh}.
}
$$
We consider the pull-back ${{\rm source}}^*\Omega_{\gh}$ of the Hodge bundle under the source map. This vector bundle is endowed with the automorphism
\begin{eqnarray*}
\sigma^* : \left((\Ch\to C, \sigma), u\right) &\mapsto & \left((\Ch\to C, \sigma), \sigma^*u\right).
\end{eqnarray*}
Thus, as in~\eqref{decompo}, we have the decomposition 
\be\label{decompo'}
{{\rm source}}^*\Omega_{\gh}=\bigoplus_{k=0}^{n-1}\Lambda^{(k)},
\ee 
where $\Lambda^{(k)}$  is the eigenbundle corresponding to the eigenvalue $\rho^k=e^{\frac{2\pi\sqrt{-1} k}{n}}$.

 In the previous section we have constructed an embedding $i:P\Mgn[\mathbf{1}]\hookrightarrow \overline{\rm Hur}^n_g$ and, by construction,  the pull-back $i^* \Lambda^{(k)}$ is, indeed, isomorphic to the Prym-Tyurin bundle $\Lambda^{(k)}$ over $P\Mgn[\mathbf{1}]$ as defined in the previous section.
 
 We see that the Prym-Tyurin vector bundles, and therefore the Prym-Tyurin classes, have a natural extension to the compactification of $P\Mgn[\mathbf{1}]$ by admissible covers. We would like, however, to extend the Prym-Tyurin classes to a different compactification of $P\Mgn[\mathbf{1}]$, namely, to $P\Mgno$. To do that we will construct a bigger space with a projection to both $\overline{\rm Hur}^n_g$ and $P\Mgno$ and use the push-forward of a pull-back.

\subsection{Space of admissible differentials}\label{sec:adm}

\begin{definition}\label{def:admdiff}
Let
$$
I : \Mgn[\mathbf{1}] \hookrightarrow \Mgno \times_{\Mb_{g}} \overline{\rm Hur}^n_g
$$ 
be the product of the two natural embeddings. The {\em moduli space of admissible $n$ differentials} $X(g,n)$, is the Zariski closure of the image of $I$ in $\Mgno \times \overline{\rm Hur}^n_g$.
\end{definition}

\begin{remark}\label{remincidence}
The space of admissible $n$-differentials is a compactification of the stratum $\Mgn[\mathbf{1}]$. In~\cite{Grushevsky-Moeller2016}, 
the authors introduced and described another compactification, the {\em incidence variety compactification}.  The incidence variety is obtained by replacing the space $\overline{\rm Hur}^n_g$ by the space $\Mb_{g,N}/S_N$ in Definition~\ref{def:admdiff}. We denote the incidence variety by $X^{\rm inc}_{g,n}$. 

There is a birational and finite map from $X_{g,n}$  to $X^{\rm inc}_{g,n}$ , however this map is not an isomorphism (see Example 4.3 of~\cite{Grushevsky-Moeller2016}).  We will make use of  $X^{\rm inc}_{g,n}$ once in this text (proof of Lemma~\ref{P3}). 
\end{remark}

The space $X(g,n)$ has two natural morphisms: ${\rm adm} :X(g,n)\to  \overline{\rm Hur}^n_g$ and ${\rm diff} : X(g,n)\to P\Mgno$. We summarize the notation on the following diagram.
$$
\xymatrix@C=3.5em{
&\bigoplus \Lambda^{(k)}\ar[d] \ar[r]^{\;\;{\rm source}^*} & \Omega_{\gh} \ar[d] \\
X(g,n)\ar[r]^{\rm adm} \ar[d]_{\rm diff} & \overline{\rm Hur}^n_g \ar[r]^{\rm source} & \Mb_{\gh} \\
P\Mgno
}
$$
By construction the forgetful map ${\rm diff} :X(g,n)\to P\Mgno$ is birational and its restriction to $P\Mgn[\mathbf{1}]$ is an isomorphism onto its image. In general the space $X(g,n)$ is not normal; thus the push-forward of classes in the Picard group under ${\rm diff}$ is ill-defined. However we will prove the following proposition in Section~\ref{cyclicsec}.

\begin{proposition}\label{pr:open}
There exist two smooth open-dense substacks $j:V\hookrightarrow X(g,n)$ and $j':U\hookrightarrow P\Mgno$ fitting into the commutative diagram
$$
\xymatrix{
V \ar[d]\ar[r]^{j \;\;\;\;} & X(g,n) \ar[d]^{\rm diff}\\ 
U\ar[r]_{j' \;\;\;\;}& P\Mgno
}
$$
such that
\begin{itemize}
\item the map $V\to U$ is an isomorphism on the underlying coarse spaces;
\item the complement of $U$ is of codimension at least  $2$ in $P\Mgno$.
\end{itemize}
\end{proposition}

The existence of $U$ and $V$ as above allows one to define the induced push-forward ${\rm diff}_*$ in the Picard groups. Indeed, the first property ensures that $\Pic(U)$ and $\Pic(V)$ are isomorphic. Similarly, the second propery ensures that $j'$ induces an isomorphism between $\Pic (U)$ and $\Pic (P \Mgno)$. Thus, given an element in $\Pic (X_{g,n})$, one defines its push-forward to $\Pic (P \Mgno)$ by first taking its pull-back by $j$, then the push-forward from $V$ to $U$, and finally the push-forward by $j'$. 

\begin{definition}
The {\em Prym-Tyurin class} $\pt^{(k)}$ on $P\Mgno$ is defined by the formula  
$\pt^{(k)}={\rm diff}_* c_1 (\Lambda^{(k)})$.
\end{definition}

\subsection{Statement of the results}

The main result obtained in this paper is the expression for the Prym-Tyurin classes and the class $\delta_{\rm deg}$ of Definition~\ref{Def:degenerate} in the 
$(\lambda,\psi,\delta_i)$ basis.

\begin{theorem}\label{mainth} 
In the rational Picard group of $P\Mgno$ we have
\be
\delta_{{\rm deg}} = 12n(n+1)\l- 2 (g-1)(2n+1)\psi- n(n+1)\sum_{i=0}^{[g/2]}\delta_i;
\label{Hodgeint}
\ee
\be\label{PTL}
\pt^{(n-k)}= (6k^2+6k+1)\lambda -\f{g-1}{n}k(2k+1)\psi- \f{1}{2}k(k+1) \sum_{i=0}^{[g/2]}\delta_i+ c_{k}\ \delta_{\rm deg},
\ee
where
\be
c_{k}=\left\{ 
\begin{array}{cl} 
\frac{2k-n}{2n},& \text{if  $(n-1)/2 < k < n$},\\[5pt]
0, & {\rm otherwise.}
\end{array}
\right. .
\label{correction}
\ee
\end{theorem}

\subsection{Strategy of the proof} 

Formula~\eqref{Hodgeint} of Theorem~\ref{mainth} is proved in two distinct ways.

\begin{itemize}
\item In Section~\ref{Bergmansec} we introduce the Bergman tau function on the moduli space $\Mgn$. We study its transformation property and its asymptotic behavior at the boundary divisors $D_{\rm deg}$ and $D_i,\;0\leq i \leq [ g/2 ]$.  We explicitly compute the vanishing order of the Bergmann tau function along these divisors. We use these results to express the divisor $\delta_\deg$ in the $(\lambda, \delta_i, \psi)$ basis of the Picard group. This first proof is a further development of the ideas introduced in~\cite{MRL} and~\cite{Contemp}.
\item In Section~\ref{sec:alternative} we give an alternative proof of Formula~\eqref{Hodgeint} based on algebro-geometric computations as introduced in \cite{Sau} and \cite{Zvon} in the context of abelian differentials. We consider the moduli space of $n$-differentials on genus $g$ curves with one marked point. This space carries a vector bundle of 2-jets of an $n$-differential at the marked point. The Euler class of this vector bundle has a natural expression involving the locus 
$A_2$ of $n$-differentials with a double zero at the marked point. The locus $A_2$ pushes forward to $D_\deg$ under the forgetful map that forgets the marked point. This allows one to compute the cohomology class $\delta_\deg$ that is Poincar\'e dual to the divisor class of $D_\deg$.
\end{itemize}

To prove Formulas~\eqref{PTL} and \eqref{correction} we combine (\ref{Hodgeint}) with the following two facts. 
\begin{itemize} 
\item First, the well-known Mumford formula \cite{Mumford} expressing the first Chern class of the vector bundle of $k$-differentials on $\Mb_{g}$ via the Hodge class:
\be
\lambda_k=(6k^2-6k+1)\lambda-\f{k(k-1)}{2}\sum_{i=0}^{[g/2]}\delta_i \;.
\la{Mum}\ee
\item Second, the fact that the morphism
\begin{equation}
\begin{array}{rcl}
\Phi_k: \Lambda^{(k)} \otimes T^{\otimes n-k} &\to&  H^0(C,\omega_C^{n-k+1}) \\
(q, v^{n-k}) &\mapsto& qv^{n-k}\;
\end{array}
\label{quv}
\end{equation}
is actually an isomorphism of vector bundles outside $D_{\rm deg}$. 
\end{itemize}

The second fact allows one to compute the rank of the Prym-Tyurin vector bundle $\Lambda^{(k)}$:
\be
{\rm rk} \, \Lambda^{(k)}  = {\rm rk}\, \Omega^{(n-k+1)}=(2n-2k+1)(g-1)\;,\hskip0.7cm k=1,\dots,n-1\;.
\la{dimOk}
\ee
It also implies that 
\be\label{expected}
\pt^{(k)}= \lambda_{n-k+1} -\f{g-1}{n}(n-k)(2n-2k+1) \psi + {\rm const}\cdot\delta_{\rm deg}\,.
\ee
In Section~\ref{PTND}, we study the asymptotic of the determinant of $\Phi_k$ along the divisor $D_{{\rm deg}}$ to obtain Expressions~(\ref{PTL}) and~\eqref{correction}.

\begin{remark} Presence of an additional contribution proportional to $\delta_{\rm deg}$ in (\ref{PTL}) for $k > (n-1)/2$ was first suggested by the third author in~\cite{Zun} using an idea of~\cite{Zvon}. 
\end{remark}

\paragraph{\bf Plan of the paper.} In Section~\ref{cyclicsec} we prove Proposition~\ref{pr:open} and thus complete the definition of the Prym-Tyurin classes. In Sections~\ref{Bergmansec} and~\ref{sec:alternative} we prove Formula~\eqref{Hodgeint} of Theorem~\ref{mainth} using the two different approaches described above. In Section~\ref{PTND} we discuss the relationship between Prym-Tyurin vector bundles and vector bundles of $k$-differentials and derive a relationship between corresponding determinant line bundles. This allows us to express the Prym-Tyurin classes in the $(\lambda, \delta_i, \psi)$ basis of the Picard group and complete the proof of Theorem~\ref{mainth}.

\section{Space of admissible $n$-differentials}\label{cyclicsec}

In this Section we justify the definition of the Prym-Tyurin classes by proving the following extended version of Proposition~\ref{pr:open}. 

\begin{proposition}\label{pr:can}
There exist two smooth open-dense substacks $j:V\hookrightarrow X(g,n)$ and $j':U\hookrightarrow P\Mgno$ fitting into the commutative diagram
$$
\xymatrix{
V \ar[d]\ar[r]^{j \;\;\;\;} & X(g,n) \ar[d]^{\rm diff}\\ 
U\ar[r]_{j' \;\;\;\;}& P\Mgno
}
$$
such that
\begin{itemize}
\item $V$ and $U$ contain the image of $P\Mgn[\mathbf{1}]$ under the embeddings into $X(g,n)$ and $P\Mgno$;
\item the complement of $U$ is of codimension at least  $2$ in $P\Mgno$;
\item the map $V\to U$ is an isomorphism on the underlying coarse spaces;
\item there exists a line bundle $T\to V$ such that $T$ is a sub-vector bundle of $\Lambda^{(1)}$, $T^{\otimes n}\simeq L$ and the restriction of $T$ to $P\Mgn[\mathbf{1}]$ coincides with the sub-vector bundle of ${\hat{\nu}}^* \Omega_{\gh}$ spanned by $v$. 
\end{itemize}
\end{proposition}

\subsection{Distinguished local coordinates on a cyclic covering}\label{sec:dist}

Consider a curve $C$ endowed with an $n$-differential $\nd$ with simple zeros. Let $f:\Ch\to C$ be the associated covering and $v =\nd^{1/n}$ be the canonical abelian differential on $\Ch$. Denote by $x_i \in C$, $i=1,\dots, N=2n(g-1)$, the branch points of $\Ch \to C$,  which coincide with the zeros of  $\nd$. Denote by $\xh_i$ the unique preimage of $x_i$ in~$\Ch$. Here we describe a specific parametrization of the covering curve $\Ch$ near the branching points of $f:\Ch\to C$. It is easy to see that all zeros of the holomorphic 1-form $v$ are situated at the ramification points $\xh_i$ and have multiplicity $n$. In other words,
\be
(v)=n \xh_1+\dots+n \xh_{N}.
\ee 

Introduce a local parameter $\zeta_i$ in a neighborhood of $x_i \in C$ and a local parameter $\xi_i$ in a neighborhood of $\xh_i \in \Ch$ such that 
\be
\nd=\zeta_i (d\zeta_i)^n, \qquad \xi_i(\xh)^{n+1}=\int_{\xh_i}^{\xh} v.
\la{distpar}
\ee
Both parameters are defined up to an $(n+1)$st root of unity and we make one choice in such a way that
$$
\zeta_i=\left(\f{n+1}{n}\right)^{n/(n+1)} \xi_i^n.
$$
The local parameters $\xi_i$ on $\Ch$ and $\zeta_i$ on $C$ given by~\eqref{distpar} are called {\it distinguished}.

Since $f^*v=\rho\, v$,  the local parameter $\xi_i(x)$ 
transforms   under the action of $f$ as $\xi_i(f(x))=\rho \xi_i(x)$. 

\subsection{Extension of the Prym-Tyurin bundles to codimension 1 loci}\label{sec:ext}

By construction, $P\Mgn[\mathbf{1}]$ is an open dense substack of $X(g,n)$. Thus ${\rm diff}:X(g,n)\to P\Mgno$ is birational and its restriction to $P\Mgn[\mathbf{1}]$ is an isomorphism onto its image.

By abuse of notation we denote by $D_i$, $0 \leq i \leq [g/2 ]$, the preimage in $P\Mgno$ of the boundary divisor $D_i \subset \Mb_g$. Then the complement of $P\Mgn[\mathbf{1}]$ in $P\Mgno$ is the union of the divisors $D_i$ and $D_\deg$. 

For each of these divisors we define a dense open locus $\widetilde{D}\subset D$ as follows.
\begin{itemize}
\item $\widetilde{D}_{0}$ is the locus of $(C,\nd)$ such that the curve $C$ has exactly one non-separating node and the differential $\nd$ has poles of order $n$ at the node and $N$ simple zeros;
\item $\widetilde{D}_{i}$ for $i\geq 1$ is the locus of $(C,\nd)$ such that the curve $C$ has exactly one separating node and the differential $\nd$ has poles of order $n$ at the node and $N$ simple zeros.
\item $\widetilde{D}_{\rm deg}$ if $(g,n)\neq(2,2)$ is the locus of $(C,\nd)$ such that the curve $C$ is a smooth curve and the differential $\nd$ has one zero of order exactly $2$ and its other zeros are simple;
\item $\widetilde{D}_{\rm deg}$ if $(g,n)=(2,2)$ is the disjoint union of the locus $\widetilde{D}_{\rm deg}(2,1,1)$ described above and of the locus $\widetilde{D}_{\rm deg}(2,2)$ of pairs $(C,\nd)$ where $C$ is smooth and $\nd$ is a square of a holomorphic differential with simple zeros.
\end{itemize}
We define $U$ as the union of $P\Mgn[\mathbf{1}]\cup \widetilde{D}_{\rm deg} \cup_{i} \widetilde{D}_i \subset P\Mgno$. We define $V$ as ${\rm diff}^{-1}(U) \subset X(g,n)$. We will prove that $U$ and $V$ satisfy the properties of Proposition~\ref{pr:can}.

Property 1 is satisfied by construction.

\begin{lemma}[Property 2]
The stack $U$ is an open substack of $P\Mgno$ and its complement is of codimension at least 2.
\end{lemma}

\begin{proof} The complement of $U$ is a union of closed substacks: the strata of curves with at least two nodes and the strata $P\Mgno[\mathbf{k}]$ for all $\mathbf{k}$ except $(1, \dots, 1)$ and $(2,1\ldots,1)$. Thus $U$ is an open substack in $P\Mgno$ and by dimension count its complement is of codimension at least~2.
\end{proof}

\begin{lemma}[Property 3]\label{P3}
The restriction of ${\rm diff}: V \to U$ induces an isomorphism on the underlying schemes. Moreover the map of stacks ${\rm diff}:V \to U$ is of degree one over $U\setminus D_{\rm deg}$ and of degree $1/2$ over $D_{\rm deg}$.
\end{lemma}

\begin{proof}
The underlying scheme of $U$ is smooth and the map ${\rm diff} : V\to U$ is birational. Thus the map ${\rm diff}$ of schemes is an isomorphism if and only if it is finite. We consider the incidence variety $X^{\rm inc}_{g,n}$ compactification defined in Remark~\ref{remincidence}. We have two birational map: ${\rm diff}' X^{\rm inc}_{g,n}\to X_{g,n}$ and $\epsilon: X^{\rm inc}_{g,n}\to \Mgno$ such that ${\rm diff}={\rm diff}' \circ \epsilon$. The map ${\rm diff}'$ is obtained by forgetting the admissible covering (but not the markings) and $\epsilon$ is obtained by forgetting the markings. As we have already stated in Remark~\ref{remincidence}, the map ${\rm diff}'$ is finite because ${\overline{\rm Hur}}_{g,n}\to \Mb_{g,N}/S_N$ is finite. Therefore we need to check that the map $\epsilon$ restricted to $\epsilon^{-1}(U)$ is finite.

The restriction $\epsilon: \epsilon^{-1}(U)\to U$ is a bijection. Indeed, if $(C,\nd)$ be a $n$-differential in $U\setminus D_{\rm deg}$, then the preimage of $(C,\nd)$ under $\epsilon$ is the $n$-differential $\nd$ with the marked simple zeros. Now if $(C,\nd)$ is a $n$-differential in $\tilde{D}_{\rm deg}$ then the preimage of $(C,\nd)$ is the point $(C',\nd',x_i)$ where $C'$ is the curve with two components: one component isomorphic to $C$ and one rational component attached to $C$ at the double zero; the differential $\nd'$ is then given by $\nd$ on the main component and vanishes identically on the rational component; finally the marked points are the simple zeros on the main component and two marked points on the rational component. 

Therefore  $\epsilon: \epsilon^{-1}(U)\to U$ is finite and ${\rm diff} : V\to U$ is birational and finite thus an isomorphism of the underlying schemes. Moreover the restriction of the map $\epsilon$ to $U\setminus D_{\rm deg}$ is obviously an isomorphism of stacks. The degree of ${\rm diff}$ along $D_{\rm deg}$ will be  computed in the next paragraphs.
\end{proof}

\begin{lemma}[Property 4]\label{P4}
There exists a line bundle $T\to V$ such that $T^{\otimes n}\simeq L$, $T$ is a sub-vector bundle of $\Lambda^{(1)}$, and the restriction of $T$ to $P\Mgn[\mathbf{1}]$ coincides with the sub-vector bundle 
of ${\hat{\nu}}^* \Omega_{\gh}$ spanned by $v$.
\end{lemma}
The proof requires a detailed analysis of the inverse morphism ${\rm diff}^{-1}: U \to V$ and is contained in the next two subsections. 

\subsubsection{Nodal curves} 

Let $0\leq i \leq [ g/2]$ and let $(C,\nd)$ be a point in $\widetilde{D}_i$. The $n$-fold covering associated to $(C,\nd)$ is given by $\Ch=\{(x,v)\in T^*_C / v^n=\nd\}$ and the canonical differential $v$ is still defined by $v(x,v)=v$. We can describe the topology of $\Ch$ and the singularities (zeros and poles) of $v$:
\begin{itemize}
\item If $i=0$, then the curve $\Ch$ is an irreducible curve with $n$ self-intersections. The differential $v$ has zeros of order $n$ at the marked points and poles of order $1$ at the nodes.
\item If $i\geq 1$ then the $\Ch$ has two irreducible components intersecting at $n$ distinct nodes. The differential $v$ has also zeros of order $n$ at the marked points and poles of $1$ at the nodes.
\end{itemize}
The canonical differential $v$ is well-defined on $U\setminus D_{\rm deg}$, thus the line bundle $T$ can be extended to ${\rm diff}^{-1}(U\setminus D_{\rm deg})$.

\subsubsection{Degenerate differentials}

Here we describe the local structure of the stacks $X(g,n)$ and $\Mgno$ close to $D_{\rm deg}$. This allows us to explicit the isomorphism of Lemma~\ref{P3} and to describe the fiber of the canonical line bundle along $\widetilde{D}_{\rm deg}$.  Let $(C_0,\nd_0)$ be a point in $\widetilde{D}_{\rm deg}$ and let $W$ be a neighborhood of $(C_0,\nd_0)$ in $D_{\rm deg}$. We will give a local parametrization of $X(g,n)$ and $\Mgno$ around the point $(C_0,\nd_0)$.
\begin{itemize}

\item {\em Parameters of $\Mgno$.} A neighborhood of $(C_0,\nd_0)$ is given by $W\times \Delta$ where $\Delta$ is a disk of $\C$ centered at zero. A point $(u, a)$ in $W\times \Delta$ parametrizes an $n$-differentials $(C,\nd)$ such that
$$
\nd=(\zeta^2+a)d\zeta^n.
$$
where the parameter $\zeta$ of the curve $C$ is uniquely determined by the choice of $a$. The parameter $a$ is a tranverse local parameter of $D_{\rm deg}$ in $\Mgno$.

\item {\em Parameters of $X(g,n)$.} A neighborhood of $(C_0,\nd_0)$ in $X_{g,n}$ is parametrized by $W\times \Delta'/(\Z/2\Z)$ where $\Delta'$ is a disk of $\C$ centered at zero.  Indeed, suppose first that the two colliding zeros of a differential $(C,\nd)$ are labeled $x_1$ and $x_2$.  Let $\zeta$ be a local parameter of $C$ such  that the positions of $x_1$ and $x_2$ are given by $\zeta_1$ and $\zeta_2$, respectively. To fix $\zeta$ uniquely we can define it by exact relation:
$$
 \nd(x)=(\zeta(x)-\zeta_1)(\zeta(x)-\zeta_2) (d\zeta(x))^n\;,\hskip0.7cm x\in C.
$$
The parameter $(\zeta_1-\zeta_2)/\{\pm 1\}$ is a local transverse parameter to $D_{\rm deg}$ in $X(g,n)$ (See Lemma~\ref{coorddeg} for a proof).
\end{itemize}

With these two local parametrizations, the map ${\rm diff}:X(g,n)\to \Mgno$ is given by
\begin{eqnarray*}
W\times \Delta'/(\Z/2\Z) &\to& W\times \Delta\\
(u,\zeta_1-\zeta_2)&\mapsto& (u,(\zeta_1-\zeta_2)^2).
\end{eqnarray*} 
This map is indeed an isomorphism of the underlying schemes.  However it is of degree $1/2$ along $D_{deg}\subset X(g,n)$ once we consider the stack structures of $X(g,n)$ and $\Mgno$. This finishes the proof of Lemma~\ref{P3}.
\bigskip
 
Finally we describe the extension of the canonical line bundle $T$ to $D_{\rm deg}$. Let $(C,\nd)$ be a family of differentials with simple zeros which tends to $(C_0,\nd_0)\in D_{\rm deg}\subset X(g,n)$. Once again we label the two coalescing zeros $x_1$ and $x_2$ and we use the local parameter of the curve $\zeta$ with $\zeta(x_i)=\zeta_i$ and $ \nd(x)=(\zeta(x)-\zeta_1)(\zeta(x)-\zeta_2) (d\zeta(x))^n $

 As $(C,\nd)$ tends to a pair $(C_0,\nd_0)$   the zeros $x_1$ and $x_2$ tend to the double zero $x_0$ of $\nd_0$. The limit curve $C_0$ is a nodal curve with two components: a Riemann sphere $C_1$  which gets naturally equipped with the {\em meromorphic} $n$-differential  $\nd_1(\zeta)$ given  by the formula
\be
\nd_1(\zeta)=(\zeta-\zeta_1)(\zeta-\zeta_2) (d\zeta)^n\;,\hskip0.7cm \zeta\in C_1\,,
 \la{ndloc1}
\ee
which is holomorphic outside of $\zeta=\infty$ and has two simple zeros at $\zeta_1$ and $\zeta_2$. The Riemann surface  $C_2$ is equipped with the holomorphic $n$-differential $\nd_0$.  The nodal point on $C_0$ is formed by identifying the point $\zeta=\infty$ on $C_1$ with the point $x_0$ on $C_2$.
 
The limit $n$-differential $(C_0,\nd_0)$ determines a canonical $n$-sheeted covering $\Ch_0\to C_0$. The curve $\Ch_0$ consists of two components $\Ch_1$ and $\Ch_2$.  The canonical covering $\Ch_1$ of $C_1$ is given by
equation
\be
v_1^n=\nd_1(\zeta)
\la{Ch1}
\ee
which, if we write $v_1=y d\zeta$, is  the curve
\be
y^n=(\zeta-\zeta_1)(\zeta-\zeta_2)
\la{Ch10}\ee
 of genus $\gh_1=[ (n-1)/2 ]$. The canonical covering $\Ch_2$ of $C_2$ is defined by the equation
\be
v_2^n=\nd_2\;;
\la{Ch2}
\ee
 its genus equals $\gh_2=\gh-[ n/2 ]$. 
 Therefore, for odd $n$ we have $\gh=\gh_1+\gh_2$ while for even $n$ we have $\gh=\gh_1+\gh_2+1$.
 
 The difference between the case of even $n$ and the case of odd $n$ is due to the fact that for odd $n$ the coverings $\Ch_1$ and $\Ch_2$ intersect at only one nodal point while for even $n$ the nodal point on $C_0$ has two pre-images on $\Ch_0$ i.e. for even $n$   $\Ch_1$ and $\Ch_2$ intersect at two nodal points,  $x_0^{(1)}$ and  $x_0^{(2)}$ (see Figure~\ref{M11} below).
 \begin{figure}[h!]
\begin{center}
\includegraphics[scale=0.45]{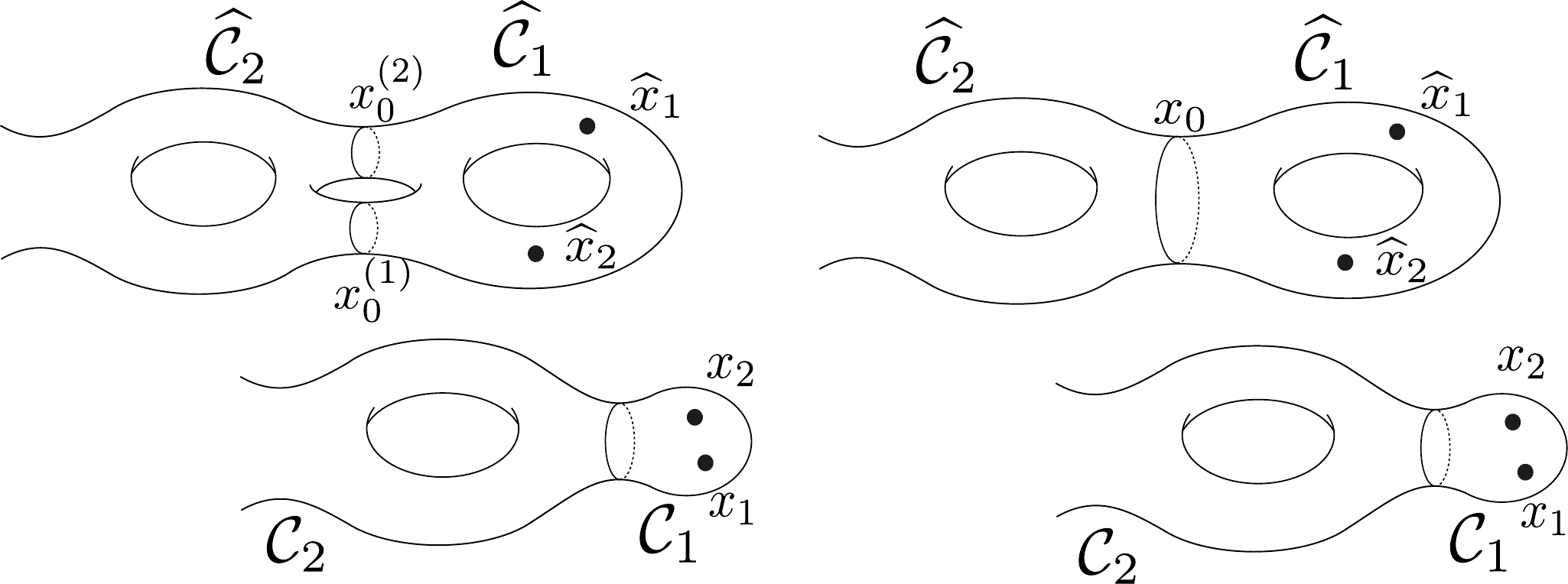}
\end{center}
\caption{Examples of degeneracy of canonical $n$-coverings (before colliding the zeros): on the left $n=4$ and on the right $n=3$. }\label{M11}
\end{figure} 
 
With this description of the limit covering, we define the limit canonical differential $v_0$ as follows: it is given by $v_0=v_2$ on the component $\Ch_2$ and vanishes identicaly on the rational component $\Ch_1$. It satisfies $v_0^n=f^*\nd_0$. Thus the canonical line bundle $T$ can be extended to the open set $V$. This completes the proof of Lemma~\ref{P4}. 

\section{Bergman tau function and Hodge class on $P\Mgno$}\la{Bergmansec}

Tau functions play an important role in the theory of integrable systems providing canonical generators for commuting
flows on the phase space \cite{BBT}. In some cases tau functions carry interesting algebro-geometric information,
like the isomonodromic tau function of the Riemann-Hilbert problem that is relevant in the theory
of Frobenius manifolds \cite{D}. 

The Bergman tau function introduced in \cite{IMRN} allowed to express the Hodge
class on the space of admissible covers of the projective line as an explicit linear combination of the boundary 
divisors \cite{Adv}. Then in \cite{KG1} this result was proven by pure algebro-geometric methods (namely, by means of the
Grothendieck-Riemann-Roch theorem) and later in \cite{KG} used to answer a question of Harris-Mumford \cite{HM} about 
the classes of Hurwitz divisors in the moduli space $\Mb_g$ of stable complex algebraic curves of even genus $g$. 

A version of the Bergman tau function for the moduli space of holomorphic abelian differentials on algebraic
curves \cite{JDG} allowed to get new relations in the rational Picard group of this space
and was applied to the Kontsevich-Zorich theory of Teichm\"uller flow \cite{MRL}, see also \cite{EKZ}. In \cite{Contemp}, the Bergman tau function was used to express the Prym class on the moduli space of holomorphic quadratic differentials in terms of the standard generators. Here we continue with developing these ideas further for the moduli space of holomorphic $n$-differentials.

\subsection{Bergman tau function on strata of $n$-differentials}

We begin with defining the Bergman tau function for each stratum $\Mgn [\bk]$ where $\bk=(k_1,\ldots,k_m)$ is a partition of $N=2n(g-1)$.
We introduce the following notation (see \cite{Fay} for precise definitions):
\begin{itemize}
\item $v_1,\dots,v_g$ -- the normalized basis of holomorphic abelian differentials with respect to a given Torelli marking (or cut system) on $C$;
\item $\O$ -- the corresponding period matrix;
\item $\Theta(z,\Omega)$ -- the theta function associated with $\O$;
\item $W(x)$ -- the Wronskian determinant of differentials $v_1,\dots,v_g$;
\item $\tilde{C}$ -- the fundamental polygon corresponding to the chosen cut system on $C$;
\item $E(x,y)$ -- the prime form on $C\times C$;
\item $\Acal_{x}$ -- the Abel map corresponding to the initial point $x$;
\item $K^{x}$ -- the vector of Riemann constants.
\end{itemize}

The distinguished local parameters on $C$ in a neighborhood of the points $x_i$ 
(zeroes of the $n$-differential $w$) are given by 
\be
\zeta_i(x)= \left(\int_{x_i}^x v\right)^{n/(k_i+n)}\,
\la{distlp}
\ee
where $k_i$ is the order of $x_i$
(in terms of these parameters $\nd\sim \zeta_i^{k_i}(d\zeta_i)^n$  near $x_i\in C$, and $v=w^{1/n}\sim  \zeta_i^{k_i/n}d\zeta_i$
near $\hat{x}_i=f^{-1}(x_i)\in\Ch$).
Then for the prime form $E(x,y)$ on $C\times C$ we have
$$
E(x,y)=\frac{E(\zeta(x),\zeta(y))}{\sqrt{d\zeta}(x)\sqrt{d\zeta}(y)}\,,
$$ 
and we put
\begin{eqnarray*}
E(\zeta,x_k)&=&\lim_{y\rightarrow x_k}E(\zeta(x),\zeta(y))\sqrt{\frac{d\zeta_k}{d\zeta}}(y),\\
E(x_k,x_l)&=&\lim_{\stackrel{\scriptstyle x\rightarrow x_k}{y\rightarrow x_l}}E(\zeta(x),\zeta(y))
\sqrt{\frac{d\zeta_k}{d\zeta}}(x)\sqrt{\frac{d\zeta_l}{d\zeta}}(y)\,.
\end{eqnarray*}
We define two vectors $Z,\,Z'\in\f{1}{n}\Z^{g}$ by the condition
\be
\f{1}{n}{\mathcal A}_{x}((\nd))+2K^x=\O Z+Z'\;.
\ee

\begin{definition}
The Bergman tau function on the space $\Mgn [\bk]$ is given by
\begin{align}\la{taudef}
&\tau(C,\nd)=\\\nonumber 
&c(x)^{2/3} e^{-\f{\pi}{6} \langle\O {Z},{Z}\rangle-\f{2\pi\sqrt{-1}}{3}\langle{Z},K^x\rangle}
\left(\f{\nd(x)}{\prod_{i=1}^{m}E^{k_i}(x,x_i)}\right)^{(g-1)/3n} 
\prod_{i<j} E(x_i,x_j)^{\f{k_i k_j}{6n^2}}\,, 
\end{align}
where
$$
c(x)=\frac{1}{W(x)}\left(\sum_{i=1}^g v_i(x)\frac{\partial}{\partial z_i}\right)^g
\theta(z;\O)\Big|_{z=K^{x}}
$$
\end{definition}

\begin{proposition}\label{symp}
Under the change of Torelli marking on $C$
$$
\left(\begin{array}{c} \tilde{{\boldsymbol b}} \\ \tilde{{\boldsymbol a}} \end{array}\right)=\left(\begin{array}{cc} A & B\\ C & D \end{array}\right) 
\left(\begin{array}{c} {\boldsymbol b}\\ {\boldsymbol a} \end{array}\right)\,,\quad \left(\begin{array}{cc} A & B\\ C & D \end{array}\right)\in Sp(2g,\Z)\,,
$$
the tau function (\ref{taudef}) transforms  as follows:
\be
\frac{\tau(C,\nd, \{\tilde{a}_i,\tilde{b}_i\})}{\tau(C,\nd, \{{a}_i,{b}_i\})}= \epsilon\, {\rm det} (C\O+D)
\ee
where $\e$ is a root of unity of degree $48d$ with $d={\rm l.c.m.}(k_1+n,\dots,k_m+n)$.
\end{proposition}

The proof can be obtained by using standard transformation properties of all factors in (\ref{taudef}) 
under the change of Torelli cut system on $C$ (cf. \cite{Fay}). The root of unity appears due to an ambiguity in the definition of the distinguished local parameters (\ref{distlp}), which translates into an ambiguity in the definition of $E(x,x_i)$ and $E(x_i,x_j)$. 
The appearance of the term ${\rm det} (C\O+D)$ can also be seen from variational formulas 
for $\tau(C,w)$ discussed below, similarly to \cite{JDG,MRL,Contemp}.

\begin{proposition}\label{homo}
The tau function has the following quasi-homogeneity property:
\be
\tau(C,\delta\nd)= \delta^\kappa\tau(C,\nd)
\la{homog}
\ee
with
\be
\ka= \frac{1}{12 n^2}\sum_{i=1}^{m}\f{k_i(k_i+2n)}{k_i+n}.
\la{homco}
\ee
\end{proposition}

This proposition follows from the explicit formula (\ref{taudef}), but can also be derived by applying the Riemann bilinear identity to variational formulas 
for $\tau$ as in was done in \cite{Adv} in the context of Hurwitz spaces.

Combining Propositions \ref{symp} and \ref{homo}, we arrive at the following
\begin{theorem}\label{holo}
On the stratum $\Mgn[{\mathbf 1}]\subset\Mgno$ of $n$-differentials with simple zeroes, the power $\tau^{48n(n+1)}$ of the tau function $\tau=\tau(C,w)$ is a nowhere vanishing holomorphic section of the line bundle $\lambda^{48n(n+1)}\otimes L^{-8(g-1)(2n+1)}\longrightarrow\Mgn[{\mathbf 1}]\,.$
\end{theorem}

In order to find the divisor of the section $\tau^{48n(n+1)}$ on $\Mgno$, we will compute the asymptotics of $\tau$ at the boundary divisors
$D_{\rm deg}$ and $D_j,\;j=0,1,\ldots,[ g/2 ]$. For that we need to study the tau function more carefully.

\subsection{Homological coordinates and variational formulas for the tau function}

The tau function $\tau(C,w)$ satisfies a system of linear differential equations on the space $\Mgn$ similar
to the tau functions on Hurwitz spaces, or spaces of abelian or quadratic differentials \cite{JDG,Adv,MRL,Contemp}.
Here we assume that all zeros of $\nd$ are simple i.e. that all $k_i=1$ in (\ref{taudef}).

The homology group $H_1(\Ch,\C)$ 
can be decomposed into the eigenspaces of the automorphism $\sigma_*$: 
$$
H_1(\Ch,\C)=\bigoplus_{i=0}^{n-1} \Hc_k\,,
$$
where $\dim \Hc_0=2g$ and in the case of simple zeros  the dimensions of $\Hc_k$ for $k=1,\dots,n-1$ are independent of $k$ and
\be
\dim\, \Hc_k= (2n+2)(g-1)\;, \quad k=1,\dots,n-1
\la{dimHk}
\ee
The dimensions (\ref{dimHk}) can be computed as the dimensions of the dual spaces $\Hc^k$ in cohomology
of $\Ch$, where $\Hc^k$ is the subspace of $H^1(\Ch,\R)$ corresponding to eigenvalue $\rho^k$.
The space $\Hc^k$ can be decomposed as $\Omega^{(k)}\oplus \overline {\Omega}^{(n-k)}$ (since 
$\rho^k=\overline{\rho}^{n-k}$) and, using (\ref{dimOk}), we get $\dim\,\Hc^k= (2n+2)(g-1)$.

For any two classes $s_1\in \Hc_{l}$ and $s_2\in \Hc_{k}$ we have $s_1\circ s_2=0$ unless $k+l=n$.
The spaces $\Hc_k$ and $\Hc_{n-k}$ are, therefore, dual to each other with respect to the standard intersection pairing
(the space $\Hc_0$ can be identified with $H_1(C)$, and, therefore, it is self-dual). 
On the other hand, for any $q\in \Omega^{(k)}$ and $s\in \Hc_l$ we have $\int_s q=0$ unless $k=l$.
In particular, since $v=w^{1/n}\in \Omega^{(1)}$, it can have non-trivial periods only over the cycles representing homology classes in $\Hc_1$.
Actually, $\Hc_1$ can be naturally identified with the tangent space to the moduli space $\Mgn$. 
Choosing a basis $\{s_i\}_{i=1}^{(2n+2)(g-1)}$ in $\Hc_1$ we introduce {\it homological coordinates}
$\Pcal_i$ on $\Mgn$ by the formula 
\be
\Pcal_i=\int_{s_i}v
\la{Homco}
\ee
(see also Corollary 2.3 of \cite{Grushevsky-Moeller2016}).

Choose a Torelli marking on $C$ and define the associated canonical bimeromorphic differential $B(x,y)=d_x d_y\log E(x,y)$ on $C$
(here $E(x,y)$ is the prime form). The bidifferential $B$ is symmetric with a second order pole with biresidue $1$ on the diagonal $x=y$ and vanishing $a$-periods
with respect to both arguments.

The bidifferential $B(x,y)$ has the following local behaviour near the diagonal $x=y$:
\be
B(x,y)= \left(\f{1}{(\zeta(x)-\zeta(y))^2}+\f{1}{6}S_B(\zeta(x))+\dots\right)d\zeta(x)d\zeta(y\,;
\la{Bhatdiag}
\ee
here $\zeta(x)$ is a local parameter, and $S_B(x)$ is the so-called Bergman projective connection.


To study the tau function on $\Mgn$ we will need some variational formulas for $B(x,y)$.
For a basis $\{s_i\}_{i=1}^{(2n+2)(g-1)}$ in $\Hc_{1}$ consider the dual basis
$\{s_j^*\}_{j=1}^{(2n+2)(g-1)}$ in $\Hc_{n-1}$, so that $s_j^*\circ s_i=\delta_{ij}$.

Consider a fundamental polygon $\tilde{C}$ of $C$ (that is, dissect $C$ along the cuts representing the Torelli marking).  
Choose a system $\Gamma$ of non-intersecting cuts that lie within $\tilde{C}$ and connect the first zero $x_1$ with other zeros of $w$. 
Pick a connected component of $f^{-1}(\tilde{C}\setminus\Gamma)\subset\Ch$ and identify it with $\tilde{C}\setminus\Gamma$. 
On $\tilde{C}\setminus\Gamma$ introduce the coordinate
\be
z(x)=\int_{x_1}^x v\,,
\la{flatco}
\ee
where the path connecting $x$ with $x_1$ entirely lies in $\tilde{C}\setminus\Gamma$.

\begin{theorem} \la{varBhteo}
The following variational formula holds for $i=1,\dots,(2n+2)(g-1)$:
\be
\f{\p}{\p \Pcal_i} B (z(x),z(y)) = \f{1}{2\pi \sqrt{-1} n}\int_{s_i^*} \f{B(z(x),\cdot) B(\cdot,z(y))}{v} 
\la{varBh}
\ee
\end{theorem}

Formula (\ref{varBh}) can be derived from the variational formulas for the stratum $\mathcal{H}_{\hat g}(n,\ldots,n)$ of the moduli space of 
holomorphic 1-differentials of genus $\gh=n^2(g-1)+1$, see Theorem 3 of \cite{JDG}, 
in a way similar to Lemma 5 of \cite{Contemp} and Proposition 3.2 of \cite{BKN}. 

Consider the differential operator  
$S_v=\frac{v''}{v}-\frac{3}{2}\left(\frac{v'}{v}\right)^2$ (that is, 
the Schwarzian derivative of the abelian integral $\int^x v$
with respect to the coordinate $z$ on $C$). 
For the holomorphic 1-differential $v$, $S_v$ is a meromorphic projective connection on $C$, 
so that the difference $S_B-S_v$ is a meromorphic quadratic differential. 

The tau function $\tau=\tau(C,w)$ satisfies the following system of differential equations with respect to the homological coordinates $\Pcal_i$ on $P\Mgn$:
\be
\frac{d}{d\Pcal_i}\log\tau=-\f{1}{12\pi \sqrt{-1}n}\int_{s_i^*}\f{ S_B-S_v}{v}
\la{deftau1}
\ee
(notice that the differential $(S_B-S_v)/v$ gets multiplied by $\rho^{-1}$ under the action of $f^*$; thus its integral over $s_i^*\in \Hc_{n-1}$ is non-trivial).
The compatibility of the system (\ref{deftau1}) follows from the symmetry of the expression
$$
\f{\p}{\p \Pcal_j}\int_{s_i^*}\f{ S_B-S_v}{v}=\f{1}{12\pi \sqrt{-1} n} \int_{s_i^*}\int_{s_j^*}\f{B(z(x),z(y))B(z(y),z(x))}{v(z(x))v(z(y))}
$$
under the interchange of $i$ and $j$.

\subsection{Asymptotics of tau function near the boundary of $\Mgno$}

Here we compute the asymptotics of $\tau$ near the boundary $\Mgno\setminus\Mgn[{\bf 1}]$ of $\Mgno$ that consist of the following divisors:
\begin{itemize}
\item $D_{{\rm deg}}$, the divisor of $n$-differentials with multiple zeroes,
\item $D_0$, the divisor of $n$-differentials on irreducible nodal curves, and
\item $D_j,\; j=1,\ldots,[ g/2 ]$, the divisors of $n$-differentials on reducible nodal curves.
\end{itemize}

\subsubsection{Coalescing simple zeros of $\nd$: divisor $D_{{\rm deg}}$}

\begin{lemma}\label{coorddeg}
Let $x_1$ and $x_2$ be two zeros of $\nd$ coalescing at $D_{{\rm deg}}$. Then 
a transversal local coordinate on $\Mgno$ in a tubular neighbourhood of $D_{{\rm deg}}$ is given by
\be
t_{{\rm deg}}=\left(\int_{x_1}^{x_2}v \right)^{2n/(n+2)}\;.
\la{tdeg}
\ee
\end{lemma}
\begin{proof} It is parallel to the proof of Lemma 8 of \cite{Contemp}.
Denote by $\zeta$ a local coordinate in a small disk $U$  containing the coalescing  zeros $x_{1,2}$ and no other zeros.
Then we can write in $U$ 
\be
\nd(\zeta)=(\zeta-\zeta(x_1))(\zeta-\zeta(x_2)) (d\zeta)^n\,,
\la{ndloc}\ee
so that
$$
\int_{x_1}^{x_2}v = 
\int_{\zeta(x_1)}^{\zeta(x_2)} \left((\zeta-\zeta(x_1))(\zeta-\zeta(x_2))\right)^{1/n} d\zeta = {\rm const}\cdot(\zeta(x_1)-\zeta(x_2))^{(n+2)/n}\,.
$$
and the parameter $t_{{\rm deg}}$ defined by (\ref{tdeg}) looks like
Since $(\zeta(x_1)-\zeta(x_2))^2$ is independent of labeling of zeroes, $t_{{\rm deg}}$ is a coordinate transversal to $D_{{\rm deg}}$.
\end{proof}

\begin{lemma}
The tau function $\tau(C,\nd)$  has the following asymptotics near $D_{{\rm deg}}$:
\be
\tau(C,\nd)= t_{{\rm deg}}^{\f{1}{12n(n+1)}} \tau(C_0,\nd_0) (1+o(1))
\la{astaudeg}
\ee
where $(C_0,\nd_0)\in D_{{\rm deg}}$.
\end{lemma}

\begin{proof} The asymptotics (\ref{astaudeg}) can be derived by computing the asymptotics of all factors in the explicit formula (\ref{taudef}).  
Alternatively, using the system of equations  (\ref{deftau1}) we see that in the limit $t_{{\rm deg}}\to 0$ the tau function $\tau(C,\nd)$ behaves like $t_{{\rm deg}}^p \tau(\C_0,\nd_0)$ for some power $p$, where $\nd_0$ is a differential with one zero of order two and all other zeroes simple. To find 
$p$ explicitly, we look at the transformation properties of $\tau$, $\tau_0$ and $t_{{\rm deg}}$ under the rescaling $\nd\mapsto\delta\nd$. 
The homogeneity degrees $\kappa$ of $\tau$ and $\kappa_0$ of $\tau_0$ are given by the formula (\ref{homco}), so that
$$
\kappa-\kappa_0= \f{1}{6n(n+1)(n+2)}
$$
On the other hand, the local parameter $t_{{\rm deg}}$ has homogeneity degree $2/(n+2)$, which gives 
$p=\f{1}{12n(n+1)}$.
\end{proof}

\subsubsection{Asymptotics of $\tau$ near $D_0$}

Take two loops $a$ and $b$ on $C$ intersecting transversally at one point; their homology class we will also denote by $a,\,b\in H_1(C,\Z)$. 
Let us pinch $a$ to a point, then $C$ degenerates to a nodal curve $C_0$ that we represent by a smooth curve of genus $g-1$ with two points (say, $x_0$ and $y_0$)
identified (we assume that all zeros of $\nd$ remain far from the node).
The holomorphic $n$-differential $\nd$ on $C$ degenerates to a meromorphic $n$-differential $\nd_0$ on $C_0$ with poles of degree $n$ at 
$x_0$ and $y_0$ such that the corresponding $n$-residues differ by $(-1)^n$.


The canonical covering $\Ch$ of $C$ degenerates to a nodal curve $\Ch_0$  
with $n$ nodes that can be thought of as $n$ pairs of points $x_0^{(m)}=\sigma^m(x_0)$ and  $y_0^{(m)}=\sigma^m(y_0),\;m=0,\dots,n-1$, on the normalization of $\Ch_0$ that are pairwise  identified (here $\sigma$ is the covering automorphism of $\Ch_0$).
The differential $v$ on $\Ch$ degenerates to a meromorphic differential $v_0$ on $\Ch_0$ with simple poles at 
the preimages of nodal points with residues at $x_0^{(m)}$ and  $y_0^{(m)}$ that differ by a sign.


Choose one of $n$ simple loops on in the preimage $f^{-1}(a)\subset\Ch$. Let us assume that this
loop pinches to the first node on $\Ch_0$.

Consider the classes $\a,\b\in\Hc_{1}$ given by
\be
\a=\sum_{m=0}^{n-1} \rho^{-m}\sigma_*^{-m} a\;,\quad
\b=\sum_{m=0}^{n-1} \rho^{-m}\sigma_*^{-m} b
\ee
(where $\sigma$ is the covering automorphism of $\Ch$),
and introduce the homological coordinates $\Pcal_\a=\int_{\a} v$ and
$\Pcal_\b=\int_{\b} v$ associated with $\a$ and  $\b$.

A local coordinate on $\Mgn$ transversal to $D_0$ in a tubular neighborhood can be chosen as
$$
t_0=e^{2\pi i \Pcal_\b/\Pcal_\a}
$$
(notice that ${\rm Im} (\Pcal_\b/\Pcal_\a)>0$ near $D_0$).

We may assume that $\Pcal_\a$ remains constant under the degeneration of $C$ to $C_0$.
Let $\omega_{x,y}$ be the abelian differential of the 3rd kind on $\Ch_0$ with simple poles at points $x$ and $y$ of residues $+1$ and $-1$ respectively, normalized with respect to some canonical basis $(a_i,b_i)$ in $H_1(\Ch_0,\Z)$. Since $v_0\in \mathcal{H}^{1}$, it can be written as
$$
v_0=\f{\Pcal_\a}{2\pi \sqrt{-1}}\, \sum_{j=0}^{n-1} \rho^m \omega_{\sigma^m(x_0),\,\sigma^m(y_0)}+ {\rm{holomorphic\; terms}}\,.
$$

In the limit $t_0\to 0$ the bidifferential $B(x,y)$ on $C\times C$ tends to the meromorphic bidifferential $B_0(x,y)$ on $C_0\times C_0$
with the same properties.

To find the asymptotics of the tau fuction $\tau$ as $t_0\to 0$ (i.e. $\Pcal_\b \to \infty$), consider the equation
\be
\f{\p \log \tau}{\p \Pcal_\b} =-\f{1}{12 \pi\sqrt{-1} n}\int_{\b^*} \f{S_B-S_v}{v} 
\underset{t_0\to 0}{\longrightarrow} -\f{1}{6n} {\rm res}_{x_0}\f{S_{B_0}-S_{v_0}}{v_0}\;,
\la{derPb}
\ee
where $\b^*=\f{1}{n}\sum_{m=0}^{n-1} \rho^{m}\sigma_*^{m} a \in\mathcal{H}_{n-1}$ is the class dual to $\b\in\mathcal{H}_1$.

To compute the residue, choose a local coordinate $\zeta$ near $x_0$ on $C_0$ such that $S_{B_0}=0$. Then we have
$v_0=\f{\Pcal_\a}{2\pi\sqrt{-1}n}\f{d\zeta}{\zeta}+ O(1)$ and
$$
\left\{\int v_0,\zeta\right\}= \left(\f{v_0'}{v_0}\right)' -\f{1}{2}\left(\f{v_0'}{v_0}\right)^2 =\f{1}{2\zeta^2}+ O(1)\,
$$
as $t_0\to 0$, so that
$
{S_{v_0}}/{v_0}= \f{\pi \sqrt{-1} n}{\Pcal_\a}\f{d\zeta}{\zeta}+ O(1)$\;.
Therefore, (\ref{derPb}) implies
\be
\f{\p\log\tau}{\p \Pcal_\b}\Big|_{t_0=0}= \f{\pi \sqrt{-1} }{6\Pcal_\a}
\ee
and $\tau\sim e^{\pi \sqrt{-1}\Pcal_\b/6\Pcal_\a}$, i.e.
\be
\tau= t_0^{1/12} ({\rm const} + o(1))
\la{asympD0}
\ee
as $t\to 0$.  

\subsubsection{Asymptotics of $\tau$ near $D_j$}

Contracting a homologically trivial simple loop $\gamma$ on $C$ we get a reducible nodal curve $C_0$ that splits into two irreducible components
$C_1$ and $C_2$ of genera $g_1=j$ and $g_2=g-j$ respectively, $j=1,\ldots,[ g/2 ]$. Denote by $x_0\in C_1$ and $y_0\in C_2$ 
the intersection point of $C_1$ and $C_2$ (the node of $C_0$). The holomorphic $n$-differential $w$ on $C$ degenerates to a pair of meromorphic $n$-differentials
$w_1$ and $w_2$ on $C_1$ and $C_2$ respectively, with poles of order $n$ at $x_0\in C_1$ and $y_0\in C_2$ whose $n$-residues differ by $(-1)^n$ (we assume that under generation the zeroes of $w$ stay away from the vanishing cycle $\gamma$).

Denote by $f_i:\Ch_i\to C_i$ the canonical $n$-fold covering ($i=1,2$), and let $x_0^{(1)},\ldots,x_0^{(n-1)}$ (resp. $y_0^{(1)},\ldots,y_0^{(n-1)}$) denote the 
preimages of the node in $\Ch_1$ (resp. $\Ch_2$) that are cyclically ordered relative to the covering maps $\sigma_i:\Ch_i\to\Ch_i$. 
The canonical cover $\Ch_0$ of the nodal curve $C_0$ is obtained from $\Ch_1$ and $\Ch_2$ by identifying 
$x_0^{(m)}$ with $y_0^{(m)}$ for each $m=0,\dots,n-1$.

Define the 1-form $v_i$ on $\Ch_i$ by $v_i^n=f_i^*\nd_i,\;(i=1,2)$.
The (meromorphic) 1-forms $v_1$ and $v_2$ have first order poles at the $n$ preimages of the node whose residues differ by a sign,
i.~e. ${\rm res}|_{x_0^{(m)}} v_1=-{\rm res}|_{y_0^{(m)}} v_2$.
Moreover, applying $m$ times the covering map $\sigma_0$, we see that ${\rm res}|_{x_0^{(m)}} v_1=\rho^{-m}{\rm res}|_{x_0^{(0)}} v_1$ 
(resp. ${\rm res}|_{y_0^{(m)}} v_2=\rho^{-m}{\rm res}|_{y_0^{(0)}} v_2$).

The preimage $f^{-1}(\gamma)\subset\Ch$ of the loop $\gamma$ on $C$ is the disjoint union of $n$ loops 
 $\gamma_m,\;m=0,\ldots,n-1$ (we enumerate them in such a way that $\gamma_{m+1}=\sigma(\gamma_m)$, assuming that $\gamma_{n} = \gamma_0$).
Note that the union of $\gamma_m,\;m=0,\ldots,n-1,$ is homologically trivial on $\Ch$.
Consider also a simple loop $\eta_0$ on $\Ch$  such that $\gamma_0\circ \eta_0=1$, $\gamma_1\circ \eta_0=-1$, and $\gamma_k\circ \eta_0=0$ for $k=2,\ldots,n-1$,
where $\circ$ denotes the intersection pairing of 1-cycles on $\Ch$.

Introduce the loops $\eta_m=\sigma^m(\eta_0),\;m=1,\ldots,n-1$, and consider the cycles
\be
\a=\sum_{m=1}^{n-1} (\rho^{-m}-1) \gamma_m\;,\quad
\b=\f{1}{\rho-1}\sum_{m=1}^{n-1} (1-\rho^{-m}) \eta_m 
\ee
(for homology classes in $H_1(\Ch,\Z)$ represented by $\gamma_m$ and $\eta_m$ we use the same notation); clearly, $\a,\b\in \Hc_1$. 
The class $\beta^*\in \Hc_{n-1}$ dual to $\b$ is given by
\be
\tilde{\a}=\f{1}{n}\sum_{m=1}^{n-1} (\rho^m-1) \gamma_m\,.
\ee

Introduce the following homological coordinates:
\be
\Pcal_\a= \int_{\a} v= n \int_{\gamma_0} v  \;,\quad
\Pcal_\b=\int_{\b} v= \f{n}{1-\rho}\int_{\eta_0} v\;.
\ee
Without loss of generality we may assume that while $C$ degenerates to $C_0$ all homological coordinates except $\Pcal_\b$ remain finite.
 
\begin{lemma}\la{locparDj}
A local parameter tranversal to $D_j\subset\Mgno$ can be chosen as
\be
t_j=e^{2\pi \sqrt{-1} \Pcal_\b/\Pcal_\a}\;.
\la{tj}
\ee
\end{lemma}
\begin{proof} We can realize the loops $\gamma_m$ and $\eta_m$ by simple closed geodesics in hyperbolic metric on $\Ch$. Denote by $T_i$ 
the maximal hyperbolic cylinder with waist $\gamma_i$ (the collar of $\gamma_i$). Put $\eta_0^{(i)}=\eta_0\cap T_i,\;i=0,1$. Then $\f{1}{n}\Pcal_\b=\f{1}{1-\rho}\int_{\eta_0}v\sim \int_{\eta_0^{(0)}}v$ when $C$ approaches $C_0$ is the
``complex heght" of the cylinder $T_0$ while $\f{1}{n}\Pcal_\a=\int_{\gamma_0} v$ is its ``complex waist". Therefore, (\ref{tj}) gives a local coordinate transversal to $D_j$.
\end{proof}

To find the asymptotics of $\tau$ when $t_j\to 0$, consider
\begin{align}\la{deryauj}
\f{\p\log\tau}{\p \Pcal_\b}&= -\f{1}{12\pi \sqrt{-1}n}\int_{{\b}^*}\f{S_B-S_v}{v}\\\nonumber &=-\f{1}{12\pi \sqrt{-1}n}\int_{\gamma_0} \f{S_B-S_v}{v}
\longrightarrow -\f{1}{6n} {\rm res}|_{x_0^{(0)}} \f{S_{B_0}-S_{v_0}}{v_0}\;.
\end{align}
Pick a coordinate $\zeta$ near $x_0^{(0)}$ such that $S_{B_0}=0$, then $v_0=\f{\Pcal_\a}{2\pi \sqrt{-1}n} \f{d\zeta}{\zeta}+O(1)$
and $\f{S_{v_0}}{v_0}=\f{\pi \sqrt{-1} n}{\Pcal_\a}\f{d\zeta}{\zeta}+O(1)$ as $t_j\to 0$.
Therefore, (\ref{deryauj}) implies 
$$
\frac{\partial\log\tau}{\partial \Pcal_\b}\Big|_{t_j=0} =\frac{\pi \sqrt{-1}}{6\Pcal_\a}\;.
$$
Thus, $\tau\sim e^{\pi \sqrt{-1}\Pcal_\b/6\Pcal_\a}$ as $t_j\to 0$, 
and
\be\label{asj}
\tau = t_j^{1/12}(const+ o(1))\;.
\ee

\subsection{Hodge class on $P\Mgno$}

A straightforward combination of Theorem \ref{holo} with asymptotic formulas \eqref{astaudeg}, \eqref{asympD0} and \eqref{asj} yields

\begin{theorem}\label{th:deg} {\rm (Formula~\eqref{Hodgeint} of Theorem~\ref{mainth})}
The Hodge class $\lambda$ on $\Mgno$ is a linear combination of the tautological class $\psi$ and the classes
of boundary divisors $D_{{\rm deg}}$ and $D_j,\;j=0,1,\ldots,[ g/2 ],$ as follows:
\be
\l=\f{(g-1)(2n+1)}{6n(n+1)}\psi +\f{1}{12n(n+1)} \delta_{{\rm deg}}+\f{1}{12}\sum_{j=0}^{[ g/2 ]}\delta_j\,.
\la{LMbar}
\ee
\end{theorem}

\section{An alternative computation of $\delta_{\rm  deg}$}\label{sec:alternative}

An alternative proof of Theorem~\ref{th:deg} was given in \cite{Zun} using an approach proposed by D. Zvonkine \cite{Zvon} and developed in~\cite{Sau}.

Let $g,n\geq2$. In order to compute the class $\delta_{\rm deg}$ in $\Pic(P\Mgno)$ we begin with marking one point on $C$, {\em i.e.}
we study the space $P\overline{\mathfrak{M}}^{(n)}_{g,1}$. In $P\overline{\mathfrak{M}}^{(n)}_{g,1}$ we define the loci
$$
Z_i=\{(C,x_1,\nd) \;  | \; 
\text{$C$ is smooth and $x_1$ is a zero of $\nd$ of order at least $i$}\}.
$$
We denote by $\overline{Z}_i$ the closure of $Z_i$. This is a closed substack of $P\overline{\mathfrak{M}}^{(n)}_{g,1}$ of pure codimension~$i$ for $1 \leq i < N=(2g-2)n$, while for $i =N$ it has two components of codimensions $N-1$ and $N$ respectively, cf Section~\ref{Ssec:strata}.

Let $\pi: P\overline{\mathfrak{M}}^{(n)}_{g,1}\to P \Mgno$ be the forgetful map of the marked point. 
Then it is easy to see that the image of $\overline{Z}_2$ under $\pi$ is the divisor~$D_{\rm deg}$. This statement takes into account the multiplicities of the components of $D_{\rm deg}$. Indeed, if $(g,n)\neq (2,2)$ then the restriction of $\pi$ to $Z_2$ is of degree~$1$ and if $(g,n)=(2,2)$ then $\pi$ is of degree one onto $P\Mgno(2,1,1)$ and two onto $P\Mgno(2,2)$. Therefore $\pi_*[\overline{Z}_2]=\delta_{\rm deg}$. Thus to find an expression of $\delta_{\rm deg}$ it suffices to compute the class $[\overline{Z}_2]\in A^2(P\overline{\mathfrak{M}}^{(n)}_{g,1})$. 
\bigskip

\noindent{\bf Computation of $[\overline{Z}_2]$.} Let $\mathcal{L}_1\to \Mb_{g,1}$ be the line bundle whose fiber is the cotangent line to the curve $C$
at $x_1$, and put $\psi_1=c_1(\mathcal{L}_1)$. 
Consider the following line bundle over $P\overline{\mathfrak{M}}^{(n)}_{g,1}$:
\begin{equation*}
\mathcal{O}(1)\otimes  \mathcal{L}_1^{\otimes n} \simeq {\rm{Hom}}\left(L, \mathcal{L}_1^{\otimes n}\right).
\end{equation*}
This line bundle has a natural section 
\begin{equation*}
s_1: (C,\nd) \mapsto \nd(x_1).
\end{equation*} 
In other words, $s_1$ is the 
evaluation of $\nd$ at the marked point. The class of the vanishing locus of $s_1$ is given by the first Chern class of the line bundle:
\begin{equation*}
\{s_1= 0\} \; = \; c_1\left(\mathcal{O}(1)\otimes  \mathcal{L}_1^{\otimes n}\right) \; = \; -\psi + n \psi_1.
\end{equation*}
It is easy to see that $\overline{Z}_1$ is a component of the vanishing locus $\{ s_1 =0 \}$. In the next section we will show that the vanishing locus has no other components and that the vanishing order of $s_1$ along $\overline{Z}_1$ is equal to~1. Thus we have $[\overline{Z}_1]=-\psi+n\psi_1$. 

Now we restrict to $\overline{Z}_1$ and study the line bundle  $\mathcal{O}(1)\otimes \mathcal{L}_1^{\otimes n+1}$. This line bundle has a natural section
\begin{equation*}
s_2 : (C, \nd) \mapsto \nd'(x_1).
\end{equation*}
In other words, assuming that $\nd$ vanishes at $x_1$, the section $s_2$ assigns to $\nd$ its derivative at~$x_1$. It is easy to see that $\overline{Z}_2$ is a component of the vanishing locus $\{ s_2 =0 \}$. In the next section we will show that the vanishing locus has no other components and that the vanishing order of $s_2$ along $\overline{Z}_2$ is equal to~1. Thus we have 
$$
[\overline{Z}_2 ]=(-\psi+(n+1)\psi_1)[\overline{Z}_1]=(-\psi+n\psi_1)(-\psi+(n+1)\psi_1).
$$ 

Recall that $\delta_{\rm deg}$ is the push-forward by $\pi$ of this expression. To compute this push-forward we use 
\begin{itemize}
\item $\psi$ is a pull-back under $\pi$;
\item $\pi_* \psi_1 = 2g-2$;
\item $\pi_* \psi_1^2 = \kappa_1 = 12 \lambda_1 - \sum_{i=1}^{[ g/2]} \delta_i$.
\end{itemize}
The last equality is the well-known Mumford's formula. 

Applying these equalities we get
\begin{eqnarray*}
\delta_{\rm deg} &=& \pi_*\left((-\psi+n \psi_1)(-\psi+(n+1)\psi_1) \right)\\
&=& n(n+1)\pi_*(\psi^2_1) - (2n+1)  \pi_*(\psi_1)\psi\\
& = & 12 n(n+1) \lambda_1 - n(n+1) \sum_{i=1}^{[ g/2]} \delta_i - (2n+1)(2g-2) \psi.
\end{eqnarray*}
This coincides with the expression of Theorem~\ref{th:deg}.
\bigskip

 In order to complete the proof of Theorem~\ref{th:deg}, it remains to prove that the vanishing locus of $s_1$ (respectively $s_2$) is exactly $\overline{Z}_1$ (respectively $\overline{Z}_2$) and that the vanishing order of $s_1$ and $s_2$ is $1$.

\noindent{\bf Vanishing loci of $s_1$ and $s_2$.}
Let $W$ be an irreducible divisor of $P\overline{\mathfrak{M}}^{(n)}_{g,1}$ in the vanishing locus of $s_1$. Let $k$ be the number of nodes of the curve represented by a generic point of $W$.  The vanishing locus of $s_1$ is of codimension 1 in $P\overline{\mathfrak{M}}^{(n)}_{g,1}$ thus $k=0$ or $1$. We investigate both cases.
\begin{itemize}
\item Let $k=0$. Then a dense subset of $W$ is contained in $Z_1$ and thus $W$ is a component of $\overline{Z}_1$.
\item  Let $k=1$. Then the divisor $W$ is contained in ${\tilde{\nu}}^{-1}(D)$ for some irreducible boundary divisor $D$ of the moduli space of stable curves with one marked point (we recall that ${\tilde{\nu}}: P\overline{\mathfrak{M}}^{(n)}_{g,1}\to \Mb_{g,1}$ is the forgetful map). Since the $D$ is irreducible, and $\tilde{\nu}$ is the projectivization of a vector bundle, we necessarily have $W={\tilde{\nu}}^{-1}(D)$. However there exists an $n$-differential in ${\tilde{\nu}}^{-1}(D)$ which is not identically zero on the component of marked point. Therefore, there exists a point in ${\tilde{\nu}}^{-1}(D)$ which is not in the vanishing locus of $s_1$. Thus the case $k=1$ does not occur.
\end{itemize}

To study the vanishing locus of $s_2$ we follow the same strategy. First we can check by dimension count that no irreducible component of $Z_1$ is in the zero locus of $s_2$. Now let $W$ be an irreducible divisor in the vanishing locus of $s_2$ and let $k$ be the number of nodes of the curve represented by a generic point of $W$. We have now 3 cases to study: $k=0,1$ and~$2$.
\begin{itemize}
\item Let $k=0$. Then a dense subset of $W$ is contained in $Z_2$ and thus $W$ is a component of $\overline{Z}_2$.
\item Let $k=2$. Then $W={\tilde{\nu}}^{-1}(D)$ where $D$ is a codimension 2 boundary stratum of $\Mb_{g,1}$. As above ${\tilde{\nu}}^{-1}(D)$ is not contained in the vanishing locus of $s_2$. Thus the case $k=2$ cannot occur.
\item Let $k=1$. Then W is a co-dimension 1 locus inside ${\tilde{\nu}}^{-1}(D)$ for a boundary divisor $D$ of $\Mb_{g,1}$. The generic curve has two components of genera $g'$ and $g-g'$ with $1\leq g'\leq g-1$. We assume that the marked point is carried by the component of genus $g'$. The rank of the bundle of $n$-differentials with a pole of order at most $n$ at the node is $n(2g'-2+1)>1$. Thus the divisor $D$ is not contained in the locus of differentials that vanish identically on this component. Thus $D$ is not contained in the vanishing locus of $s_2$.
\end{itemize}

We conclude that $\{s_i=0\}=\overline{Z}_i$ for $i=1$ and $2$.
\bigskip

\noindent{\bf Vanishing order of $s_1$.}  Let $y_0=(C_0,\nd_0,x_0)$ be a point in $Z_1$. We recall that $P\overline{\mathfrak{M}}^{(n)}_{g,1}\to P\Mgno$ is isomorphic the universal curve. Thus a neighborhood of $y_0$ in $Z_1$ is given by $U\times \Delta$ where $U$ is a neighborhood of $(C_0, \nd_0)$ in $P\Mgn[\mathbf{1}]$ and $\zeta \in \Delta$ is the distinguished parameter around  $x_0$ in $C_0$ (cf Section~\ref{sec:dist}). Let $(C,\nd,x)$ be an $n$-differential in $U\times \Delta$. In coordinates $(u,\zeta)\in U\times \Delta$, the differential $\nd$ is given by $\nd=\zeta d\zeta^{n}$, the locus $Z_1$ is $\{ \zeta=0 \}$ and the section $s_1$ is given by $s_1(u,\zeta)=\zeta$. Therefore the vanishing order of $s_1$ along $Z_1$ is~$1$.
\bigskip

\noindent{\bf Vanishing order of $s_2$.} Let $y_0=(C_0,\nd_0,x_0)$ be a point in $Z_2$. A neighborhood of $y_0$ in $P\overline{\mathfrak{M}}^{(n)}_{g,1}$ is now given by $U\times \Delta\times \Delta'$ where $U$ is a neighborhood of $y_0$ in $Z_2$ and $\Delta$ and $\Delta'$ are disks of the complex plane centered at 0 and parametrized by $\zeta$ and $a$ such that:
$$
\nd=(\zeta^2+a) d \zeta^n.
$$
With the parameters $(u,\zeta,a) \in U\times \Delta\times \Delta'$, the locus $Z_1$ is defined by $\zeta^2+a=0$. Moreover with these parameters, the section $s_2$ is given by $s_2(u,\zeta,a)=a$.  Thus the vanishing order of $s_2$ along $Z_2$ is again $1$.

\section{Prym-Tyurin differentials on $\Ch$ and holomorphic $n$-differentials on $C$}
\label{PTND}

Here we relate Prym-Tyurin vector bundles to vector bundles of holomorphic $k$-differentials on the base Riemann surface $C$. We use this relation to finish the proof of Theorem~\ref{mainth}. 

We also prove that the Prym-Tyurin bundle is not a pullback from $P\Mgno$ in general.

\subsection{Prym-Tyurin bundles and $n$-differentials}

Consider two vector bundles $\Lambda^{(k)}$ and $\tilde{\nu}^* \Omega_g^{(n-k+1)}$ over $X(g,n)$. The fiber of $\Lambda^{(k)}$ is the $k$th eigenspace in the space of abelian differentials on~$\Ch$. The fiber of $\Omega_{n-k+1}$ is the space of $(n-k+1)$-differentials on~$C$. There are natural morphisms:
\be
\label{isophi}
\begin{array}{ccccc}
\Phi_0  : &   \Lambda^{(0)}  & \to &  \tilde{\nu}^* \Omega_g \, , \qquad  \ & \\
\Phi_k  : &   \Lambda^{(k)} \otimes T^{\otimes (n-k)}  & \to &    \tilde{\nu}^* \Omega_g^{(n-k+1)} & \mbox{ for } 1 \leq k \leq n. \\
\end{array}
\ee
Indeed, let $(C,\nd)$ be a point in $U\setminus D_{\rm deg}$ and let $q$ be a differential in the fiber of $\Lambda^{(k)}$. The $n-k+1$ differential $qv^{n-k}$ is invariant under the action of the automorphism group of the covering, thus $qv^{n-k}$ is the pull-back of $n-k+1$ differential on $C$. For $k=0$ the differential $q$ is already invariant under the action of the automorphism group of the covering, so there is no need to multiply it by a power of~$v$.

\begin{lemma}\la{corresp}
The map $\Phi_k$ is an isomorphism over $V \setminus D_{\rm deg}$.
\end{lemma}

\begin{proof}
The maps $\Phi_k$  for $0 \leq k \leq n-1$ are injective because $v$ does not vanish identically on any component of the nodal curve $\Ch$. They are bijective because
the sum of ranks of the Prym-Tyurin bundles $\Lambda^{(k)}$ for $0 \leq k \leq n-1$ is equal to the sum of ranks of the vector bundles $\Omega_{n-k+1}$. Indeed,
$$
\sum_{k=0}^{n-1} \rk \Lambda^{(k)} = \rk \Omega_{\gh} = n^2(g-1)+1
$$
and
$$
\rk \Omega_g + \sum_{k=1}^{n-1} \rk \Omega_g^{(n-k+1)} = 
g + \sum_{k=1}^{n-1} (2n-2k+1)(g-1) = n^2(g-1)+1.
$$
\end{proof}

\begin{corollary}
The rank of the Prym-Tyurin bundle is $g$ for $k=0$ and $(2n-2k+1)(g-1)$ for $1 \leq k \leq n-1$.
\end{corollary}

\begin{corollary}
In ${\rm Pic}(P \Mgno \setminus D_{\rm deg})$ we have
\be
\pt^{(k)}=\lambda_{n-k+1} -\f{g-1}{n}(n-k)(2n-2k+1)\psi
\la{PTL1}
\ee
for $1 \leq k \leq n-1$.
\end{corollary}
 
\begin{proof}
On the locus where $\Phi_k$ is an isomorphism we have
\begin{eqnarray*}
\pt^{(k)} &= &  c_1(\Lambda^{(k)})\\
&=& c_1(\tilde{\nu}^*\Omega_k \otimes T^{\otimes -(n-k)})\\ 
&=& \lambda_{n-k+1}- (g-1)(n-k)(2n-2k+1) c_1(T)\\
&=& \lambda_{n-k+1} -\f{g-1}{n}(n-k)(2n-2k+1)\psi,
\end{eqnarray*}
where the last equality is due to $T^{\otimes n}=L$.

The locus where $\Phi_k$ is an isomorphism coincides with $P \Mgno \setminus D_{\rm deg}$ up to codimension 2 substacks that are immaterial for the first Chern class computations. 
\end{proof}
 
To study the extension of the formula (\ref{PTL1}) to $P\Mgno$ we need to study the behavior of $\Phi_k$ along the boundary divisor $D_{\rm deg}$. The determinant of $\Phi_k$ is a global section of 
 $${\rm det}(\Lambda^{(k)} \otimes T^{n-k})^{-1}\otimes {\rm det}( \tilde{\nu}^* \Omega_g^{(n-k+1)}).$$  Thus the difference between $\pt^{(k)}$ and $\lambda_{n-k+1} -\f{g-1}{n}(n-k)(2n-2k+1)\psi$ is an effective divisor defined as the vanishing locus of ${\rm det}\,\Phi_k$.
 
In Section~\ref{sec:ext}, we have described a parametrization of  the cyclic coverings along $D_{\rm deg}$. We use here this parametrization to prove the following Lemma.
 
 \begin{lemma}
 \la{lemPT}
 If $[(n-1)/2]+1\leq k \leq n-1$ or $k=0$, then the morphism $\Phi_k$ is an isomorphism of vector bundles over $V\subset X(g,n)$. Otherwise, ${\rm det}(\Phi_k)$ vanishes along $D_{\rm deg}$ with order $(1-\frac{2k}{n})$.
 \end{lemma}
 
This lemma implies the following
\begin{corollary}
The following relations between Prym-Tyurin class $\pt^{(k)}$, the class $\psi=c_1(L)$ and the pullback of class $\lambda_{n-k+1}$ from $\overline{\Mcal}_g$ to  $P\Mgno$ holds:
\begin{align}\la{PTL2}
&\pt^{(k)}=\lambda_{n-k+1} -\f{g-1}{n}(n-k)(2n-2k+1)\psi +\left(\f{1}{2}-\f{k}{n}\right)\delta_{{\rm deg}} \;,\\
&\hspace{3in} 1\leq k\leq [(n-1)/2]\,,\nonumber\\\la{PTL20}
&\pt^{(k)}=\lambda_{n-k+1} -\f{g-1}{n}(n-k)(2n-2k+1)\psi \;,\\
&\hspace{2.5in} [(n-1)/2]+1\leq k \leq n-1\,.\nonumber
\end{align}
\end{corollary}
 
This corollary together with Theorem~\ref{th:deg} and Formula~(\ref{Mum}) completes the proof of Theorem~\ref{mainth}. 
\begin{remark}
Note the difference of a factor $2$ between the vanishing order of ${\rm det}\,\Phi_k$ and the contribution of $\delta_{\rm deg}$ in $\pt^{(k)}$. This is due to the fact that $V\to U$ is of degree $1/2$ along $D_{\rm deg}$.
\end{remark}

The proof of Lemma~\ref{lemPT} will occupy the two following sections.  We consider separately the cases of even $n$ and odd $n$. 
  
\subsection{Odd $n=2m+1$}

\subsubsection{Kernel and cokernel of $\Phi_k$}

Let $(C_0,\nd_0)$ be a point in $D_{{\rm deg}}$  and $f:\Ch_0\to C_0$ be the associated admissible covering. We recall that $\Ch_0$ is a curve with two components intersecting in one point. The two components $\Ch_1$ and $\Ch_2$ are of genera $\gh_1=m$ and $\gh_2=\gh-m$ (see Section~\ref{sec:dist}). We have denoted by $\nd_1$ the meromorphic $n$-differential on $C_1$ given by $(\zeta-\zeta_1)(\zeta-\zeta_2)d\zeta^n$. The curve $\Ch_1$ is parametrized by $y^n=(\zeta-\zeta_1)(\zeta-\zeta_2)$ and the $n$th root of $f^*\nd_1$ is given by $v_1=yd\zeta$. The covering $\Ch_2\to C_2$ is defined by $v_2^{\otimes n}=\nd_0$. Finally, let $v$ be the canonical differential satisfying $v^{\otimes n}=f^*\nd_0$: it is determined by $v=v_2$ on the component $\Ch_2$ and vanishes identically on the component $\Ch_1$.

 Denote the fiber of $k$th Prym-Tyurin vector bundle $\L^{(k)}$ over the point $(C_0,\nd_0)\in D_{{\rm deg}}$ by $\Omega_0^{(k)}$. We can decompose 
 $$\Omega_0^{(k)}= \Omega_1^{(k)}\oplus \Omega_2^{(k)}$$
 where  $\Omega_1^{(k)}$ is the space of holomorphic differentials on $\Ch_1$ which get multiplied by $\rho^k$ under the action of the automorphism $(y,\zeta)\to (\rho^k y,\zeta)$; the space $\Omega_2^{(k)}$ is the analogous space of holomorphic differentials on $\Ch_2$. 
 
The canonical differential $v$ vanishes identically on $\Ch_1$. Thus the kernel of $\Phi_k$ contains space $\Omega_1^{(k)}\otimes T^{\otimes n-k}$. On another hand, the restriction of the morphism $\Phi_k$ to the linear subspace $\Omega_{2}^{(k)}\otimes T^{\otimes n-k}$ is injective. Indeed the differential $v$ does not vanish identically on $\Ch_2$.  Thus
$$
{\rm ker}\, \Phi_k\simeq \Lambda_1^{(k)}\otimes T^{(n-k)}.
$$
over a generic locus of $D_{\rm deg}$.
 
We have $\dim \,\Omega_1^{(k)}=1$ for $k=1,\dots,m$. These one-dimensional spaces are generated by the holomorphic differentials on $\Ch_1$ given by $q_1^{(k)}= \f{d\zeta}{y^{2m+1-k}}$. For $k= m+1\dots 2m+1$ the eigenspaces $\Omega_1^{(k)}$ are trivial.
  Therefore, 
 $$
 \dim \Omega_2^{(k)} = \dim \Omega_0^{(k)}-1\;,\hskip0.7cm k=1,\dots, m\;;
 $$
 $$
 \dim \Omega_2^{(k)} = \dim \Omega_0^{(k)}\;,\hskip0.7cm k=m+1,\dots, 2m.
 $$
 We can also describe the images of $\Phi_k$.  For $k=m+1,\dots,2m$, the morphism $\Phi_k$ is an isomorphism from $\Omega_{2}^{(k)}\otimes T^{\otimes n-k}$ to  $H^{0}(C_0,\omega_{C_0}^{n+k-1})$. However, for $k=1,\dots,m$ the image of $\Omega_{2}^{(k)}\otimes T^{\otimes n-k}$ is the space of holomorphic $n-k+1$ differentials vanishing at $x_0$. 

 We have proved the following
 
 \begin{lemma}
 The kernel of $\Phi_k$ over $D_{deg}$ is the vector bundle $\Omega_1^{(k)}\otimes T^{\otimes n-k}$. This kernel is trivial for $k>m$ and $k=0$. If $1\leq k\leq  m$, the image of $\Phi_k$ is the vector bundle whose fibers are the $H^0(C,\omega^{n+k-1}(-x_0))$ where $x_0$ is the unique zero of order $2$.
 \end{lemma}
Therefore the first part of Lemma~\ref{lemPT} and Formula (\ref{PTL20}) are valid for odd $n$.
 
\subsubsection{Study of $\Phi_k$ for $k=1\ldots m$}

We fix $1\leq k\leq m$. We have seen that the kernel and cokernel of $\Phi_k$ are of dimension $1$.  Let $(C_0,\nd_0)$ be a generic point in $D_{\rm deg}$. Let $W$ be an open neighborhood of $(C_0,\nd_0)$ in $X(g,n)$ with a non-vanishing section $q_0^{(k)}$ of $\Lambda^{(k)}|_{W}$ such that  $q^{(k)}_0|_{D_{\rm deg}}$ spans $\ker\,\Phi_k|_{W\cap D_{\rm deg}}$. The section $\Phi_k$ is of co-rank 1 along $D_{\rm deg}$, thus the vanishing order of ${\rm det}\,\Phi_k$ is equal to the vanishing order of $\Phi_k({q}_0^{(k)}\otimes v^{\otimes n-k})$ along $D_{\rm deg}$. Therefore, we will construct such a local section $q^{(k)}_0$ of $\Lambda^{(k)}$ and study the asymptotic behavior of $q^{(k)}_0\otimes v^{\otimes n-k}$ along $D_{\rm deg}$.
\bigskip

Let $\tilde{q}^{(k)}_0$ be
 a non-vanishing section of $\tilde{\nu}^*(\Omega^{(n-k+1)})$ over $W$ such that: for all $(C,\nd) \in D_{\rm deg}$, $q^{(k)}_0(C,\nd)$ is a differential that does not vanish at the double zero of $\nd$. Up to a choice of a smaller $W$, such a section exists. We label the coalescing zeros by $x_1$ and $x_2$. We chose the parameter of the curve $\zeta$ such that position of $x_1$ and $x_2$ are $\zeta_1$ and $\zeta_2$ and $\nd=(\zeta-\zeta_1)(\zeta-\zeta_2)(d\zeta)^n$ (see Section~\ref{sec:ext}). We define
$$
q^{(k)}_0= (\zeta_1-\zeta_2)^{-1+2k/n}\cdot\frac{f^*(\tilde{q}^{(k)}_0)}{v^k}.
$$
We recall that root $(\zeta_1-\zeta_2)^{2/n}$ is well defined, it is the integral of $v$ between $\xh_1$ and $\xh_2$ (see Lemma~\ref{coorddeg}). 

Over $W\setminus D_{\rm deg}$,  the differential $q^{(k)}_0$ is obviously a non-vanishing section of $\hat{\nu}^*(\Lambda^{(k)})$. Beside, along $D_{\rm deg}$ the differential $q^{(k)}_0$ vanishes identically on $\Ch_2$ because of the factor $(\zeta_1-\zeta_2)^{-1+2k/n}$. To compute the limit of $q^{(k)}_0$ on $\Ch_1$, we remark that $q_0^{(k)}$ is equivalent to
\begin{equation}\label{equidiff}
(\zeta_1-\zeta_2)^{1-2k/n} [(\zeta(x)-\zeta_1)(\zeta(x)-\zeta_2)]^{k/n-1}d\zeta\;.
\end{equation}
in coordinate $\zeta$. The differential~\eqref{equidiff} is invariant under simultaneous rescaling $\zeta\to \epsilon\zeta$, $\zeta_i\to \epsilon\zeta_i$, $i=1,2$.
 Therefore, as $x_{1,2}\to x_0$, the differential $q_0^{(k)}$ tends to the holomorphic differential 
 $$q_1^{(k)}=y^{-n+k} dx=
 (\zeta_1-\zeta_2)^{1-2k/n} [(\zeta(x)-\zeta_1)(\zeta(x)-\zeta_2)]^{k/n-1}d\zeta
$$
on the curve $\Ch_1$ (the generator of $\Omega_1^{(k)}$).

The image of $q_{0}^{(k)}\otimes v^{\otimes n-k}$ under $\Phi_k$ is  $(\zeta_1-\zeta_2)^{1-2k/n}\tilde{q}_0^{(k)}$ by construction. Thus the determinant of $\Phi_k$  is equivalent to a constant times $(\zeta_1-\zeta_2)^{1-2k/n}$. The parameter $(\zeta_1-\zeta_2)^2$ being a transverse parameter to $D_{\rm deg}$, we conclude that the vanishing order of $\Phi_k$ along $D_{\rm deg}$ is $1-\frac{2k}{n}$. This finishes the proof of Lemma~\ref{lemPT} for odd $n$.
 $\Box$
 
 \subsection{Even $n=2m$}
 
 The proof of Lemma~\ref{lemPT} is essentially identical to the odd case. Let $(C_0,\nd_0)$ be a generic point in $D_{\rm deg}$. Now we have $\gh_1=m-1$ and $\gh_2=\gh-m$ and the two components intersect in two points. Let $\Omega_0^{k}$ be the PT bundle over $(C_0,\nd_0)$ and let $\Omega_i^{(k)}$ be the subspace of $\Omega_0^{k}$ of holomorphic differentials supported on $\Ch_i$. We still have the decomposition:
 $$
 \Omega_0^{(k)}= \Omega_1^{(k)}\oplus \Omega_2^{(k)}
 $$
 except for $k=m$. The kernel $\Phi_k$ is  $\Omega_1^{(k)}\otimes T^{\otimes n-k}$. This kernel is trivial for $k=m+1,\dots,2m-1$. For $k=1,\dots,m-1$, the map $\Phi_k$ has co-rank one and the generator of  $\Omega_1^{(k)}$ is the differential  $q_1^{(k)}= y^{k-2m}d\zeta$ on the canonical covering $\Ch_1$ $y^{2m}=(\zeta-\zeta_1)(\zeta-\zeta_2)$ (see Section~\ref{sec:ext} for definition of the parameters).

 For $k=m$, the space $\Omega_0^{(k)}$ contains a $1$-dimension subspace of differentials of  a third kind. These are differentials with simple poles at the nodal points $x_0^{(1)}$ and $x_0^{(2)}$ (both on $\Ch_1$ and $\Ch_2$) and opposite residues. The image of such differential under $\Phi_m$ is a holomorphic $m$-differential. Beside $\Omega_1^{(m)}$ is trivial. Therefore the morphism $\Phi_k$ is also bijective for $k=m$.\bigskip
 
 Similarly to the case of odd $n$, for $k=1,\dots,m-1$, let $\tilde{q}_0^{(k)}$ be a section of $\tilde{\nu}^*(\Omega^{(n-k+1)})$ over a neighborhood of $(C_0,\nd_0)$ that does not vanish at the double zero of $\nd_0$. We define $
q^{(k)}_0= (\zeta_1-\zeta_2)^{-1+2k/n}\cdot\frac{f^*(\tilde{q}^{(k)}_0)}{v^k}.
$ and we study the asymptotic behavior of $\Phi_k( q_0^{(k)}\otimes v^{n-k})$ along $D_{\rm deg}$. 

As in the odd case, the differential $q^{(k)}_0$ is a non vanishing section of $\Lambda^{(k)}$ and $\Phi_k( q_0^{(k)}\otimes v^{n-k})=(\zeta_1-\zeta_2)^{-1+2k/n}\cdot \tilde{q}^{(k)}_0$ by construction. Therefore the vanishing order of ${\rm det}(\Phi_k)$ is given by ${1}-\frac{2k}{n}$ for $k=1,\ldots,m-1$.
 $\Box$
 
 \subsection{Obstruction to the extension of the Prym bundles to $\Mgno$}
 
In order to define the Prym-Tyurin classes we have constructed the space of admissible differentials $X(g,n)$ (see~\ref{sec:adm}). Indeed, the Prym-Tyurin bundles are naturally defined over $X(g,n)$. The following theorem explains the necessity of the introduction of the space $X(g,n)$.

\begin{theorem}
Let $g> 2$ and $k> 0$. There exists no vector bundle $\tilde{\Lambda}\to P\Mgno$ such that $\Lambda^{(k)}={\rm diff}^* \tilde{\Lambda}^{(k)}$, where ${\rm diff}:X(g,n)\to P\Mgno$ is the forgetful map.
\end{theorem}

\begin{proof}
Suppose that there exists $\tilde{\Lambda}\to P\Mgno$ such that $\Lambda^{(k)}={\rm diff}^* \tilde{\Lambda}^{(k)}$. Then in particular $\pt^{k}=c_1(\tilde{\Lambda}^{(k})$ and thus $c_1(\Lambda^{k}) = {\rm diff}^*\pt^{k}$. We will prove that this equality does not hold for $g>2$ and $k>0$.

Let $m>2$ and $\mu$ be the partition of $n(2g-2)$ given by $(m,1,\ldots,1)$. We denote by $\Mgn[m]$ the locus $\Mgn[\mu]\subset \Mgno$. We suppose that ${\rm gcd}(m,n)=1$. We denote by $D_m$ the preimage of $\Mgn[m]$ in $X(g,n)$ under ${\rm diff}$. The locus $D_m$ is a divisor whose generic points are elements $(C,\nd,x_i,f:\Ch\to C)\in X(g,n)$ such that:
\begin{itemize}
\item the curve $C$ is a nodal curve with two components: $C_2$ isomorphic to $C$ with $n(2g-2)-m$ marked points attached in one node to a rational component $C_1$ at the zero of order $m$;
\item the $n$-differential $\nd$ is identically zero on $C_1$ and has profile $\mu$ on $C_1$;
\item the covering curve $\Ch\to C$ has two component $\Ch_2$ and $\Ch_1$. The component $\Ch_2$ determined by the $\nd$ as in Section~\ref{sec:ext} and $\Ch_1\to C_1$ is the unique $n$-sheeted ramified covering maximally ramified at the marked points and the node.
\end{itemize}

Moreover the canonical root $v$ of $f^*\nd$ vanishes identically on $\Ch_1$ and has a zero of order $m+n-1$ at the preimage of the zero of order $m$. Thus the morphism $\Phi_k: \Lambda^{(k)} \otimes T^{\otimes (n-k)}  \to     \tilde{\nu}^* \Omega_g^{(n-k+1)} $ has a non-empty co-kernel along $D_m$ for $m$ large enough. Therefore the line bundle 
$${\rm det}(\Lambda^{(k)}) \otimes {\rm diff}^{*}\left({\rm det} \tilde{\Lambda}^{(k)\vee}\right)$$ 
has a global section which vanishes along divisors contained in $X(g,n)\setminus V$. Thus $\Lambda^{(k)}\neq  {\rm diff}^{*}\tilde{\Lambda}^{(k)}$. 
\end{proof}

\noindent{\bf Acknowledgements.} DK and PZ acknowledge the hospitality of the Max-Planck-Institut f\"ur Mathematik in Bonn where this work began in 2011, 
and DK also thanks the Max-Planck Institute for Gravitational Physics (Albert Einstein Institute) in Golm where this work was continued. Research of DK was supported in part by the Natural Sciences and Engineering Research Council of Canada grant RGPIN/3827-2015, by the FQRNT grant ``Matrices Al\'eatoires, Processus Stochastiques et Syst\'emes Int\'egrables'' (2013PR 166790), by the Alexander von Humboldt Stiftung and by the Chebyshev Laboratory of St.~Petersburg State University. Research of Section \ref{Bergmansec} was supported by the Russian Science Foundation grant 16-11-10039.
The authors are very grateful to Dimitri Zvonkine for his helpful advice and comments. AS is also grateful to Dawei Chen, Charles Fougeron, Martin Mo\"eller and Anton Zorich for very fruitful conversations.

\end{document}